\documentclass[10pt,twocolumn,twoside]{IEEEtran}

\usepackage{amsmath,graphicx,epsfig,color,amsfonts, algorithm, amsthm, subcaption}
\usepackage{version,xspace}
\usepackage[noend]{algorithmic}

\def\real{\mathbb{R}}

\newcommand{\subscr}[2]{#1_{\textup{#2}}}

\newcommand\numberthis{\addtocounter{equation}{1}\tag{\theequation}}

\newcommand\oprocendsymbol{\hbox{$\square$}}
\newcommand\oprocend{\relax\ifmmode\else\unskip\hfill\fi\oprocendsymbol}

\def \bs {\boldsymbol}

\newtheorem{theorem}{Theorem}

\newtheorem{corollary}[theorem]{Corollary}

\newtheorem{remark}{Remark}

\title{Multi-agent decision-making \\ dynamics inspired by honeybees
\thanks{This research has been supported in part by NSF grant CMMI-1635056, ONR grant N00014-14-1-0635 and DGAPA-PAPIIT(UNAM) grant IA105816.}}

\author{Rebecca Gray$^\star$, Alessio Franci$^\star$, 
Vaibhav Srivastava,~\IEEEmembership{Member,~IEEE}, Naomi Ehrich Leonard,~\IEEEmembership{Fellow,~IEEE}
\thanks{$^\star$These authors contributed equally to this work.}
\thanks{R. Gray and N. E. Leonard are with the Department of Mechanical and Aerospace Engineering, Princeton University, Princeton, NJ, USA,\tt{ \{rgray, naomi\}@princeton.edu}.}
\thanks{A. Franci is with the Department of Mathematics, Universidad Nacional Aut\'onoma de M\'exico, Ciudad de M\'exico, Mexico,  \tt afranci@ciencias.unam.mx}
\thanks{V. Srivastava is with the Department of Electrical and Computer Engineering, Michigan State University, East Lansing, MI, USA, \tt vaibhav@egr.msu.edu}}

\begin{document}
\maketitle

\begin{abstract}
When choosing between candidate nest sites, a honeybee swarm reliably chooses the most valuable site and even when faced with the choice between near-equal value sites, it makes highly efficient decisions. Value-sensitive decision-making is enabled by a distributed social effort among the honeybees, and it leads to decision-making dynamics of the swarm that are remarkably robust to perturbation and adaptive to change.  To explore and generalize these features to other networks, we design distributed multi-agent network dynamics that exhibit a pitchfork bifurcation, ubiquitous in biological models of decision-making.   Using tools of nonlinear dynamics we show how the designed agent-based dynamics recover the high performing value-sensitive decision-making of the honeybees and rigorously connect investigation of mechanisms of animal group decision-making to systematic, bio-inspired control of multi-agent network systems.  We further present a distributed adaptive bifurcation control law and prove how it enhances the network decision-making performance beyond that observed in swarms.
\end{abstract}

\begin{IEEEkeywords}
Adaptive control; animal behavior; bifurcation; decision-making; decentralized control; multi-agent systems; networked control systems; nonlinear dynamical systems. 
\end{IEEEkeywords}

\section{Introduction}
For many applications of multi-agent systems, ranging from  transportation networks and mobile sensing networks to power networks and synthetic biological networks, successful network-level decision-making among alternatives is fundamental to tasks that require coordination among agents.  Enabling a group of individual agents to make a single choice among alternatives allows the group to collectively decide, for example, which alternative is true, which action to take, which direction or motion pattern to follow, or when  something in the environment or in the state of the system has changed.

Since animals do so well at making collective decisions between alternatives, we look to animal groups for inspiration.   For example, a honeybee swarm  makes a single, accurate  choice of the best quality nest site among scouted-out sites \cite{SeeleyBestofN}.   A fish school  makes a single choice among potential food sources about which some individuals may have prior information or preference \cite{Couzin2011}.   Migratory birds  choose together when to depart a rest stop and continue on their route \cite{Eikenaar2014}.   
Indeed, from bird flocks to fish schools, animal groups are known to manage tasks, such as migration, foraging, and evasion, with speed, accuracy, robustness and adaptability~\cite{JKP-LEK:99}, even though individuals employ distributed strategies with limitations on sensing, communication, and computation \cite{JK-GDR:02,SumpterBook}.   

As in the case of  animal groups, successful operation of engineered networks in complex environments requires robustness to disturbance and adaptation in the face of changes in the environment.  Further, like individual animals,  agents in these kinds of networks typically use distributed control and have limitations on sensing, communication, and computation.

Mechanisms used to study collective animal behavior depend on the animals' social interactions and on their perceptions of their environment. A rigorous understanding of these dependencies will make possible the translation of the mechanisms into a systematic, bio-inspired design methodology for use in engineered networks. This remains a challenge, however, in part because most studies of collective animal behavior are empirically based or rely on mean-field models.

To address this challenge, we present a generalizable agent-based, dynamic model of distributed decision-making between two alternatives. In this type of decision-making, the pitchfork bifurcation is ubiquitous \cite{Leonard2014}; for example, it appears  in the dynamics of honeybees choosing a nest site and schooling fish selecting a food source. Our approach is to derive the agent-based model so that it too exhibits the pitchfork bifurcation.   This allows  animal group dynamics and   multi-agent  dynamics  to be  connected mathematically by mapping to the normal form of the pitchfork bifurcation.   The agent-based model provides an important complement to mean field or simulation models by allowing for rigorous analysis of the influence on system performance of distributed information, network topology, and other heterogeneities across the network.

Development of our framework has been largely influenced by the singularity theory approach to bifurcation problems \cite{Golubitsky1985}. The distinction among state, bifurcation, and the so-called unfolding parameters in that theory reflects the hierarchy among controlled variables, control variables, and model parameters in control theory. This analogy naturally translates bifurcation theory into control theory and guides the design of intrinsically nonlinear and tunable behaviors. This connection was originally made in \cite{AF-RS:14} to design neuronal behaviors; see also \cite{AF_FC_2016} for a neuromorphic engineering application.  

In the present paper we specialize the Hopfield network dynamics~\cite{JJH:84} by restricting them to a class of systems suitable for distributed cooperative decision-making, and we rigorously design and analyze distributed dynamics that achieve bio-inspired collective decision-making in an engineered network. Towards this end, we employ techniques from singularity theory, Lyapunov theory, Lyapuniv-Schmidt reduction, and geometric singular perturbation theory.

The generalizability of our framework provides a systematic way to translate a wide range of animal group dynamics into distributed multi-agent feedback dynamics; in this paper we focus the networked agent-based model dynamics on recovering the high performing, value-sensitive decision-making of a honeybee swarm choosing between two candidate nest sites as has been studied empirically and with a mean-field model \cite{Seeley12,Pais2013}; see \cite{reina2017model} for an extension to multiple alternatives and~\cite{AR-etal:15} for an agent-based approach.  Remarkably, honeybees efficiently and reliably select the highest value nest site, and for alternatives of equal value, they quickly make an arbitrary choice when the  value is sufficiently high. 

Our framework also provides the means to study and design decision-making dynamics beyond that observed in biological networks.  To enhance control of  network decision-making dynamics, we introduce an adaptive control law for a bifurcation parameter in the model.  We  design the proposed adaptive control law using a time-scale separation and leverage singular perturbation theory to prove that the adaptively controlled dynamics achieve a unanimous decision in the network.
See related work on control of bifurcations \cite{Abed1986,Chen2000,Krener2004}.

The major contributions of our work are fourfold.  First, we introduce a generalizable agent-based, dynamic  model for bio-inspired decision-making for a network of distributed agents. Second, we characterize and rigorously establish its rich bifurcation behavior for generic strongly connected networks. Robustness and adaptability of decision-making in the proposed model are discussed in terms of the uncovered bifurcation structure and singularity theory concepts. Third, for a class of symmetric networks, we use model reduction and asymptotic expansion to show further in depth how the model captures the adaptive and robust features of value-sensitive honeybee decision-making dynamics and how sensitivity to scale and heterogeneity can be explored.  Fourth, we design distributed adaptive feedback control dynamics for a bifurcation parameter that ensures a unanimous decision in the network. We characterize and rigorously establish the rich nonlinear phenomena exhibited by the adaptive dynamics.

Section \ref{sec:HoneyBees} describes the motivating example of honeybee decision-making dynamics, highlighting  value-sensitive decision-making, adaptability and robustness. In Section \ref{sec:PresentModel}, we present the agent-based decision-making model and analyze the associated nonlinear phenomena in a generic strongly connected network. Model reduction to a low-dimensional, attracting manifold in a class of symmetric networks is established in Section \ref{sec:ModelRed} and used to prove an analytical approximation to the pitchfork bifurcation point.  These analytic results are used in Section \ref{sec:Results} to show value-sensitivity, performance, and  influence of model parameters in the agent-based model.    In Section \ref{sec:feedback} we present the distributed adaptive control law for the bifurcation parameter and analyze the nonlinear phenomena in the resulting controlled adaptive dynamics. We conclude in Section \ref{sec:final}.

\section{Biological Inspiration: Value-sensitive decision-making in honeybees}
\label{sec:HoneyBees}

When a honeybee hive becomes overcrowded, a cast swarm goes out to find a new nest site that will provide sufficiently valuable storage and shelter for surviving the next winter. Honeybees in the swarm identify candidate nest sites, typically cavities in trees, and then collectively choose one among the alternatives. Swarms reliably choose the most valuable site among the alternative nest sites \cite{SeeleyBestofN}, and even when faced with the choice between near-equal value sites, they make highly efficient decisions \cite{Seeley12}. According to the model in \cite{Seeley12}, their decisions are sensitive to both the relative and average value of candidate nest sites \cite{Pais2013}.  This {\em value-sensitive decision-making} means that honeybee swarms can adapt their decisions to the value of available sites.

Collective decision-making in honeybees has been extensively studied, providing insight about not just the characteristic outcomes, but also the contributing mechanisms. From experimental work, it is known that a small population of workers, called scouts, find and assess the value of candidate nest sites using criteria that include site volume, size of entrance, and height above the ground. Each scout advertises for its site at the swarm, using a ``waggle dance" to communicate the site's location and its assessment of the site's value. The scouts also use a cross-inhibitory stop-signal to stop the dancing of the scouts recruiting for a competing site \cite{Seeley12}.

A model of the mean-field population-level dynamics of the swarm was derived in \cite{Seeley12} under the assumption that the total honeybee population size is very large. Analysis of the model was carried out in \cite{Pais2013} to rigorously explain the mechanisms that lead to value-sensivity of the decision-making dynamics. The model considers decision-making between two alternatives and the dynamics of the fractions of the total population that are committed to each of the two sites. In the case of alternatives with equal value, the dynamics exhibit a supercritical pitchfork bifurcation in which the collective decision emerges as the stop-signal rate increases. For stop-signal rate less than a critical value, the only stable solution is the deadlock state, which corresponds to equality of the two committed fractions of the population and therefore no decision.  For stop-signal rate greater than the critical value, deadlock is unstable and there are two stable solutions, corresponding to dominance of agents committed to one of the alternatives. For large enough stop-signal rate, the number of dominant agents will be greater than a quorum threshold, and the two stable solutions correspond to a decision for one of the alternatives.

Further, the critical value of the stop-signal rate is inversely proportional to the average value of the alternatives.  This means that the higher the value of the alternatives, the lower is the stop-signal rate required to break deadlock and get an arbitrary decision for one of the alternatives. If the alternatives have low value and the stop-signal rate is too low for a unanimous decision, the honeybees may be  waiting for a better alternative.  If the wait yields no new alternative, the honeybees could then ramp up the stop-signal rate until it crosses the critical value for a decision \cite{Pais2013}.  This would provide a way of adapting to their circumstances and getting a timely decision even if it less than ideal.

When asymmetry is introduced, the pitchfork unfolds to a persistent bifurcation diagram.  The unfolded bifurcation diagrams are robust in the sense that the decision behavior away from the bifurcation point resembles the symmetric case and also in the sense that the behavior close to the bifurcation point is fully determined by singularity analysis. The way in which the pitchfork unfolds close to the bifurcation point is exactly what determines the sensitivity of the decision-making to asymmetric internal and environmental parameter variations, providing a further mechanism for adaptability that complements value-sensitivity. The unfolded pitchfork bifurcation diagrams also exhibit a hysteresis that makes the group decisions robust to small fluctuations in the relative value of the alternatives.

The characteristics of \emph{value-sensitivity}, \emph{robustness} and \emph{adaptability} are highly desirable in a decision-making process, and \cite{Seeley12} and \cite{Pais2013} demonstrate how these characteristics can arise in decision-making dynamics when they are organized by a pitchfork bifurcation. However, the model in \cite{Seeley12} and \cite{Pais2013} builds on an assumption of a well-mixed (mean-field) population, and so it cannot directly be used to design distributed control strategies or to examine the influence of network topology or distribution of information across the group. This motivates the design of distributed agent-based dynamics that exhibit a pitchfork bifurcation and inherit the advantageous features of the honeybee decision-making dynamics.

\section{An agent-based decision-making model organized by a pitchfork singularity}\label{sec:PresentModel}
We  propose an agent-based decision-making model that specializes the Hopfield network dynamics (\cite{JJH:82}, \cite{JJH:84}).  The model provides generalizable network decision-making dynamics for a set of $N$ agents, and by design it exhibits a pitchfork bifurcation tangent to the consensus manifold. To describe decision-making between two alternatives A and B, let $x_i \in \mathbb{R}$, $i \in \{1, ..., N\}$, be the state of agent $i$, representing its opinion. Agent $i$ is said to favor alternative A  (resp. B) if $x_i > 0$ (resp. $x_i<0$), with the strength of agent $i$'s opinion  given by $|x_i|$. If $x_i = 0$, agent $i$ is {undecided}. The collective opinion of the group at time $t$ is defined by the average opinion $y(t) = \frac{1}{N} \sum_{i=1}^N x_i(t)$. Let $\subscr{y}{ss}$ and $\subscr{\bs x}{ss}$ be steady-state values of $y(t)$ and ${\bs x}(t) =(x_1, \ldots, x_N)^T$, respectively. As proved in Theorem~\ref{theorem:unfolding generic} below, the existence of $\subscr{y}{ss}$ and $\subscr{\bs x}{ss}$ is ensured by the boundedness of trajectories and the monotonicity of the proposed model.

Let the group's \emph{disagreement} $\delta$ be defined by $\delta = | \subscr{y}{ss}| - \frac{1}{N} \| \subscr{\bs x}{ss} \|_1$, where $\| \cdot \|_1$ is the vector $1$-norm. If each entry of $\subscr{\bs x}{ss}$ has the same sign, then there is no disagreement, i.e, $\delta =0$. We say that the group's decision-making is in \emph{deadlock} if either $\bs x_{ss}=\bs 0$ (no decision) or $\delta \ne 0$ (disagreement). A collective decision is made in favor of alternative $\rm A$ (resp. $\rm B$) if $\delta=0$ and $\subscr{y}{ss} > \eta$ (resp. $\subscr{y}{ss} < -\eta$), for some appropriately chosen threshold $\eta \in \real_{>0}$. 

The network interconnections define which agents can measure the state, that is, the opinion, of which other agents, and this is encoded using a network adjacency matrix $A \in \mathbb{R}^{N \times N}$. Each $a_{ij} \geq 0$ for $i,j \in \{1, ..., N\}$ and $i \neq j$ gives the weight that agent $i$ puts on its measurement of agent $j$. Then $a_{ij} > 0$ implies that $j$ is a neighbor of $i$ and we draw a directed edge from $i$ to $j$ in the associated graph. We let $a_{ii} = 0$ for all $i$ and $D \in \mathbb{R}^{N \times N}$ be a diagonal matrix with diagonal entries $d_i = \sum_{j = 1}^N a_{ij}$. $L = D-A$ is the Laplacian matrix of the network graph. In our illustrations, we use $a_{ij}\in \{0,1\}$.

We model the rate of change in state of each agent over time as a function of the agent's current state, the state of its neighbors, and a possible external stimulus $\nu_i$:
\begin{equation} \label{eq:DynScalar}
\frac{dx_i}{dt} = - u_I d_i x_i + \sum_{j = 1}^{N} u_S a_{ij} S(x_j) + \nu_i.
\end{equation}
$\nu_i \in \mathbb{R}$ encodes external information about an alternative received by agent $i$, or it represents the agent's preference between alternatives (we will use ``information" and ``preference" interchangeably). We let $\nu_i \in \{\nu_\text{A}, 0, -\nu_\text{B}\}$, $\nu_\text{A}, \nu_\text{B} \in \mathbb{R}^+$.   If $\nu_i = \nu_\text{A}$ (resp. $\nu_i = -\nu_\text{B}$) agent $i$ is informed about, or prefers, alternative A (resp. B).  If $\nu_i = 0$ agent $i$ receives no information or has no preference. $u_I > 0$ and  $u_S > 0$ are  control parameters and $S: \mathbb{R} \to \mathbb{R}$ is a smooth, odd sigmoidal function that satisfies the following conditions: $S'(z) > 0, \ \forall z \in \mathbb{R}$ (monotone); $S(z)$ belongs to sector $(0,1]$; and sgn$(S''(z)) = -$sgn$(z)$, where $(\cdot)'$ denotes the derivative with respect to the argument of the function, and sgn$(\cdot)$ is the signum function. In our illustrations, we use $S(\cdot) = \tanh(\cdot)$.

Control $u_I$ can be interpreted as the {\it inertia} that prevents agents from rapidly developing a strong opinion. The term $u_S S(x_j)$ can be interpreted as the opinion of agent $j$ as perceived by agent $i$. $S(x)$ is a saturating function, so opinions of small magnitude are perceived as they are, while opinions of large magnitude are perceived as saturating at some cap. Control $u_S$ represents the strength of the \emph{social effort}: a larger $u_S$ means more attention is paid to other agents' opinions.

Let $\boldsymbol{\nu} = (\nu_1, \ldots, \nu_N)^T$, and $\bs S(\bs x)= (S(x_1), \ldots, S(x_N))^T$.  Then \eqref{eq:DynScalar} can be written in vector form as
\begin{equation} \label{eq:DynVectoru1u2}
\frac{d{\bs x}} {dt}= -u_I D\boldsymbol{x} + u_S A\boldsymbol{S}(\bs{x}) + \bs{\nu}.
\end{equation}  
To simplify notation, we  study~\eqref{eq:DynVectoru1u2} using a timescale change $s = u_I t$.  We  denote ${\bs x}(s)$ by ${\bs x}$ and $d{\bs x}/ds$ by $\dot {\bs x}$.  Let $u = u_S/u_I$, $\beta_i = \nu_i /u_I$, $\beta_\text{A} = \nu_{\rm A}/u_I, \; \beta_\text{B} = \nu_{\rm B}/u_I$ and $\boldsymbol{\beta} = (\beta_1, \ldots, \beta_N)^T$. Then each $\beta_i \in \{\beta_\text{A}, 0, -\beta_\text{B}\}$ and~\eqref{eq:DynVectoru1u2} is equivalent to
\begin{equation} \label{eq:DynVector}
\boldsymbol{\dot{x}} = -D\boldsymbol{x} + uA\boldsymbol{S}(\bs{x}) + \bs{\beta}.
\end{equation}

Dynamics \eqref{eq:DynVector} are designed to exhibit a symmetric pitchfork bifurcation in the uninformed case $\bs{\beta}=\bs 0$, with the additional requirement that the two stable steady-state branches emerging at the pitchfork do so along the consensus manifold. In other words, we have designed dynamics \eqref{eq:DynVector}, equivalently dynamics \eqref{eq:DynVectoru1u2}, as a model of {\it unanimous} collective decision-making between two alternatives. To provide intuition about why these dynamics exhibit a pitchfork bifurcation along the consensus manifold, let the network graph be fixed and strongly connected.  Then rank($L$) = $N-1$ and $L\bs{1}_N = \bs{0}$, where $\bs{1}_N$ is the vector of $N$ ones. Observe that the linearization of \eqref{eq:DynVector} at $\bs{x} = \bs{0}$ for $u = 1$ and $\bs{\beta} = \bs{0}$ is the linear consensus dynamics $\bs{\dot{x}} = -L \bs{x}$. Because $L$ has a single zero eigenvalue with the consensus manifold as the corresponding eigenspace, it follows from the center manifold theorem \cite[Theorem 3.2.1]{Guckenheimer2002} that \eqref{eq:DynVector} has a one-dimensional invariant manifold tangent to the consensus manifold at the origin. On this manifold, the reduced one-dimensional dynamics undergo a bifurcation, which, by odd (that is, $Z_2$) symmetry of \eqref{eq:DynVector} with $\bs{\beta}=\bs{0}$, will generically be a pitchfork \cite[Theorem VI.5.1, case~(1)]{Golubitsky1985}.

This argument is proved rigorously in Theorem~\ref{theorem:unfolding generic} for an arbitrary, strongly connected graph and the more general $\bs{\beta}\neq\bs{0}$. We specialize to the all-to-all case in Corollary~\ref{theorem:all-to-all-unif}. Geometric illustrations are provided in Figures~\ref{FIG: pitch on consensus} and ~\ref{fig:univ-unfolding}. Concepts from singularity theory used in the proof of Theorem~\ref{theorem:unfolding generic} are defined in the monograph \cite{Golubitsky1985}, where the first chapter provides a thorough, accessible introduction to the theory, in particular in relation to control problems.  We cite specific references as needed throughout the proof of Theorem~\ref{theorem:unfolding generic}. We also provide insights into the meaning of the singularity concepts after the theorem statement and before its proof.

\subsection{Pitchfork bifurcation by design in generic network}

A preliminary version of the following theorem can be found in the preprint \cite{Franci2015}. Let $g(y, u, \bs \beta)$ be the Lyapunov-Schmidt reduction\footnote{The Lyapunov-Schmidt reduction is the projection of the vector field near its singular point along the null space of the Jacobian at the singular point. The reduction exploits the implicit function theorem to compute a local approximation of the vector field in the subspace orthogonal to the null space. The reduced dynamics can be used to infer associated bifurcations. See~\cite{Golubitsky1985} for a detailed discussion.} of~\eqref{eq:DynVector} at $(y, u, \bs \beta)=(0,1 , \bs 0)$.
\begin{theorem}\label{theorem:unfolding generic}
The following hold for the dynamics~\eqref{eq:DynVector} where the graph is fixed and strongly connected:
\begin{enumerate}
\item For $\bs \beta =\bs 0$, $\bs x = \bs 0$ is globally asymptotically stable if $0<u \le  1$, and locally exponentially stable if $0<u<1$.
\item The bifurcation problem $g(y,u,\bs 0)$ has a symmetric pitchfork singularity at $(\bs x^*, u^*) =(\bs 0, 1)$. For $u>1$ and $|u-1|$ sufficiently small the Jacobian of \eqref{eq:DynVector} at $\bs x = \bs 0$ possesses a single positive eigenvalue and all other eigenvalues are negative. The ($N-1$)-dimensional stable manifold separates the basins of attraction of the other two steady states bifurcating from the pitchfork, which attract almost all trajectories.
Further, the steady-state branches bifurcating from the pitchfork for $u>1$ are exactly the origin and $\pm y^s \bs 1_N$,  where $\{0,\pm y^s\} $ are the three solutions of the equation $y-uS(y)=0$, $u>1$.
\item For $\bs\beta\neq 0$, the bifurcation problem $g(y, u, \bs\beta )$ is an $N$-parameter unfolding of the symmetric pitchfork. Moreover $\frac{\partial g}{\partial \beta_i}(0,1,\bs 0)=\bar v_i$, where $\bar{\bs v} = (\bar v_1, \ldots, \bar v_N)$ is the null left eigenvector of $L$.  $\bar v_i$ is known as the eigenvector centrality of node $i$ in the network graph\footnote{Eigenvector centrality measures the relative influence of a node in a graph. Because eigenvectors are not unique, it is only defined up to a scaling factor.}.
\end{enumerate}
\end{theorem}

Before proving Theorem~\ref{theorem:unfolding generic}, we discuss its implications, notably in terms of robustness and adaptability of the decision-making process.  The types of robustness and adaptability described in the following remarks are not shared by linear collective decision-making models.

\begin{remark}[Pitchfork as organizing center yields deadlock breaking]
\textup{Theorem~\ref{theorem:unfolding generic} states that the model dynamics~\eqref{eq:DynVector} are organized by the pitchfork. Similar to the honeybee collective decision-making dynamics, the model~\eqref{eq:DynVector} possesses a pitchfork singularity. In the case $\bs\beta=\bs{0}$, before the bifurcation point ($u<1$), the deadlock state $\bs{x} = \bs{0}$ is globally exponentially stable. After the bifurcation point ($u>1$ and $|u-1|$ sufficiently small), the deadlock state $\bs{x} = \bs{0}$ is unstable and two symmetric decision states are jointly almost-globally asymptotically stable (Figure~\ref{FIG: pitch on consensus}). Thus, our agent-based decision-making dynamics~\eqref{eq:DynVector}  qualitatively capture a fundamental observation in the biological social decision-making that increasing social effort (control $u$) breaks deadlock and leads to a decision through a pitchfork bifurcation. Larger values of $u$ might lead to secondary bifurcations, which reflects the local nature of Theorem~\ref{theorem:unfolding generic}, at least for generic strongly connected graphs. Necessary and sufficient conditions for the appearance of secondary bifurcations under symmetric interconnection topologies are provided in~\cite{Fontan2017}.}
\end{remark}
\begin{figure}
\begin{center}
\includegraphics[width=60mm]{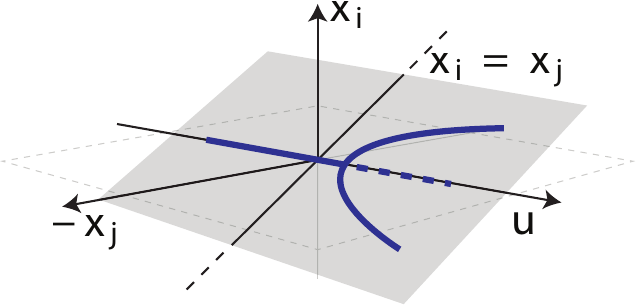}
\caption{For $u=1$ and $\bs\beta=\bs{0}$, dynamics~\eqref{eq:DynVector} exhibit a pitchfork bifurcation at $\bs{x}=\bs{0}$. The steady-state branches emerging at the singularity lie on the consensus manifold $\{x_i = x_j\, |\, i,j \in \{1, ... , N\}\}$ shown in gray.  Branches of stable and unstable solutions are shown as solid and dashed lines, respectively.}
\label{FIG: pitch on consensus}
\end{center}
\end{figure}

\begin{remark}[Robustness to perturbation and unmodeled dynamics]
\textup{It follows from singularity theory that the collective decision-making characterized in Theorem~\ref{theorem:unfolding generic} is {\em robust} in the following sense. For $\bs\beta\neq\bs0$ the symmetric pitchfork generically unfolds, that is, it disappears and breaks into qualitatively different bifurcation diagrams. However, by unfolding theory \cite[Chapter~III]{Golubitsky1985}  only the four cases depicted in Figure~\ref{fig:univ-unfolding}(a) can generically be observed (see \cite[Section~III.7]{Golubitsky1985} for a detailed development in the pitchfork singularity case). These are the persistent (that is, robust) bifurcation diagrams in the universal unfolding of the pitchfork. These and only these four bifurcation diagrams can be observed, almost surely, in any model, or actual system, organized by the pitchfork singularity. Thus, the possible qualitative behaviors are robust (in the sense of the bifurcation phenomena connecting deadlock and decision) under small perturbations and/or unmodeled dynamics. It is exactly in this sense that the pitchfork singularity is the {\em organizing center} of collective decision-making between two alternatives. A detailed analysis of the unfolded dynamics \eqref{eq:DynVector} will be presented in the upcoming work \cite{Franci2018}. Note that all four persistent bifurcation diagrams of the pitchfork can be found in the model~\eqref{eq:DynVector} with generic strongly connected graphs.}
\end{remark}

\begin{figure}[h!]
\centering
\includegraphics[width=\linewidth]{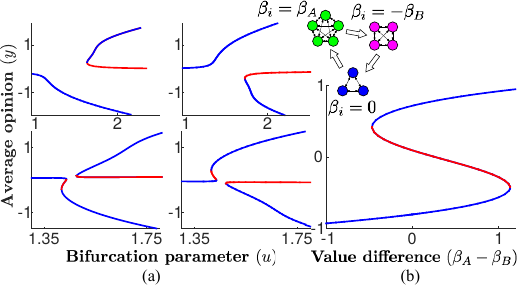}
\caption{Robustness of dynamics \eqref{eq:DynVector} illustrated for the graph shown with $N=12$. A white arrow indicates a directed edge from every node in one all-to-all connected group to another. Blue lines are stable solutions and red lines unstable solutions. (a) The four persistent bifurcation diagrams of the universal unfolding of the pitchfork. Top left: $\beta_A = \beta_B = 5$.  Top right: $\beta_A = \beta_B = -5$.  Bottom left: $\beta_A = 5$, $\beta_B = 4.5$.  Bottom right: $\beta_A = -5$, $\beta_B = -4.5$.  (b) Hysteresis in the universal unfolding of the pitchfork with $\beta_A = 1$ and $\beta_B$ varying.}\label{fig:univ-unfolding}
\end{figure}

\begin{remark}[Sensitivity to system and environmental changes] \textup{Organizing centers provide both robustness {\em and  sensitivity} to nonlinear phenomena. Infinitesimal deviations from the perfectly symmetric state, for example a slight overall preference for one of the two alternatives, will unfold the pitchfork, break deadlock, and enforce convergence to a decision. This ultrasensitive responsiveness provides the mechanism through which decision-making efficiently {\em adapts} to changes in environmental cues or agent preferences and interconnections. And this form of adaptation is robust, because no aberrant behaviors can be observed besides the four persistent bifurcation diagrams prescribed by the organizing center. The coexistence of robustness and sensitivity seems, almost tautologically, a necessary condition for survival and evolution. As in biological systems, not all decisions made through this mechanism will correspond to democratic or optimal decisions because of the influence of the network topology and agent preference strength  \cite{Couzin2011}. In an engineering setting, to ensure that the best decision is made, careful attention must be paid to the relation between system parameters and the opening of the pitchfork. This can be done using constructive sensitivity analysis at the organizing center that is grounded in unfolding theory and Lyapunov-Schmidt reduction. Theorem~\ref{theorem:unfolding generic}(iii) provides such a sensitivity analysis, showing how an agent's influence depends on its eigenvector centrality. Theorem~\ref{theorem:unfolding generic}(iii) predicts the bias in the decision in favor of the smaller group's preference $- \beta_B$, shown in Figure~\ref{fig:univ-unfolding}, because the sum of eigenvector centralities for the smaller informed group is larger than the sum of eigenvector centralities for the larger informed group.} 
\end{remark}

\begin{remark}[Robustness to small fluctuations]
\textup{Theorem~\ref{theorem:unfolding generic} implies a further type of robustness of collective decision-making organized by the pitchfork due to bistability and the associated hysteresis. All the persistent bifurcation diagrams of the pitchfork exhibit bistability between two stable steady states (the two alternatives) for sufficiently large bifurcation parameter, and this implies a hysteresis. In the collective decision-making setting this means that after a decision has been taken it is robustly maintained despite small adjustments in agent preferences. Only sufficiently large preference changes can switch the collective decision to the other alternative. This is illustrated in Figure~\ref{fig:univ-unfolding}(b). At the mathematical level, this robustness can be proved by observing that the unfolding of the pitchfork is represented by a parametrized path in the unfolding of the {\it cusp catastrophe} \cite[Section III.12]{Golubitsky1985}. By traveling this unfolding transversally to these paths by modulating unfolding parameters (instantiated here by agent preferences), we recover the described hysteresis behavior.}
\end{remark}

{\it Proof of Theorem~\ref{theorem:unfolding generic}.}
(i) Let $V(\bs x)=\frac{1}{2}\bs x^T \bs x$.  Then, for $0<u\leq 1$,
\begin{align*}
\dot V&= \bs x^\top(-D\bs x+u A \bs S(\bs x)) \\
&=\bs x^\top(-D \bs x+uD \bs S(\bs x)-u D \bs S(\bs x)+u A \bs S(\bs x))\\
&=- \bs x^\top D(\bs x-u \bs S(\bs x))- u \bs x^TL \bs S(\bs x))\\
&< -u \bs x^\top L \bs S(\bs x) \le 0,\quad \forall \bs x \neq 0, \;\; 
\end{align*}
since $uS$ is a monotone function in the sector $[0,1]$, $D$ is diagonal and positive definite, and  $L$ is  positive semi-definite.  Local exponential stability for $0<u<1$ follows since the linearization of (\ref{eq:DynVector}) at $\bs x = \bs 0$ is
\[
\dot{\delta \bs x}=(-D+uA)\delta \bs x, 
\]
and $(-D+uA)$ is a diagonally dominant and Hurwitz matrix. 

(ii) Let $F(\bs x,u,\bs\beta)=-D\bs x+uAS(\bs x)+\bs\beta$, i.e., the right hand side of dynamics \eqref{eq:DynVector}. Observe that, by odd symmetry of $F$ in $\bs x$,
$$F(-\bs x,u,\bs 0)=-\dot{\bs x}=-F(\bs x,u,\bs 0). $$
That is, for $\bs\beta=0$, $F$ commutes with the action of $-I_N$. It follows by \cite[Proposition~VII.3.3]{Golubitsky1985} that the Lyapunov-Schmidt reduction of $F$ at $(\bs x,u)=(\bs 0,1)$ is also an odd function of its scalar state variable, that is, $g(-y,u,\bs 0)=-g(y,u,\bs 0)$. To show that $g$, and therefore $F$, possesses a pitchfork bifurcation at the origin for $u=1$, it suffices to show that $g_{yyy}(0,1,\bs 0)<0$ and $g_{yu}(0,1,\bs 0)>0$.  This follows because all of the degeneracy conditions in the recognition problem of the pitchfork ($g_{yy}(0,1,\bs 0)=g_u(0,1,\bs 0)=0$) are automatically satisfied by odd symmetry of $g$, and $g(0,1,\bs 0)=g_y(0,1,\bs 0)=0$ because of the properties of the Lyapunov-Schmidt reduction \cite[Equation I.3.23(a)]{Golubitsky1985}.

Let $\bar{\bs v}\in({\rm Im}(L))^\perp$ be a null left eigenvector of $L$, and $P= I_N - \frac{1}{\sqrt{N}} \bs 1 \bar{\bs v}^\top$ be a projector on ${\rm Im}(L)=\bs 1_N^\perp$. Then, using \cite[Equation I.3.23(c)]{Golubitsky1985} it holds that 
\begin{multline*}
g_{yyy}(0,1,\bs 0)=\left\langle\bar{\bs v},d^3F_{\bs 0,1,\bs 0}(\bs 1,\bs 1,\bs 1) \right.\\ \left.
-3d^2F_{\bs 0,1,\bs 0}(\bs 1,L^{-1}Pd^2F_{\bs 0,1, \bs 0}(\bs 1,\bs 1)) \right\rangle,
\end{multline*}
where $\langle \cdot, \cdot \rangle$ denotes the inner product, and $d^k F_{y,u,\bs \beta}$ is the $k$-th order derivative defined by~\cite[Equation I.3.16]{Golubitsky1985}:
\[
d^k F_{\bs x,u,\bs \beta} (\bs v_1, \ldots, \bs v_k) = \]
\[\frac{\partial}{\partial t_1} \ldots \frac{\partial}{\partial t_k} F\left(\sum_{i=1}^k t_i \bs v_i, u, \bs \beta\right) \bigg|_{t_1=\ldots=t_k=0}.
\]

Note that $d^2F_{\bs 0,1,\bs 0}=0_{N\times N\times N}$ because $S''(0)=0$. On the other hand
$$\frac{\partial^3 }{\partial x_l x_k x_h}F_i(\bs x,u,\bs 0)=u\delta_l^k\delta_k^h\delta_h^j a_{ij}S'''(x_j), $$
which implies that $d^3F_{\bs 0,1,\bs 0}(\bs 1,\bs 1,\bs 1)_i=uS'''(0)\sum_{j=1}^N a_{ij}<0$. Since $\bar{\bs v}$ is a non-negative vector and not all entries are zero, it follows that $g_{yyy}(0,1,\bs 0)<0$.

Similarly, using \cite[Equation I.3.23(d)]{Golubitsky1985}, we have
$$g_{uy}(0,1,\bs 0)=\left\langle \bar{\bs v},d\frac{\partial F_{\bs 0,1,\bs 0}}{\partial u}(\bs 1)\right\rangle= \left\langle \bar{\bs v},\left[\sum_{j=1}^N a_{ij}\right]_{i=1}^N \right\rangle>0,$$
where we have already neglected the second-order term depending on $d^2F_{\bs 0,1,0}$, which is zero.

It follows by the recognition problem for the pitchfork \cite[Proposition~II.9.2]{Golubitsky1985} that \eqref{eq:DynVector} undergoes a pitchfork bifurcation at the origin when $u=1$. For $u>1$ and $|u-1|$ sufficiently small, there are exactly three fixed points. The origin is a saddle with an $(N-1)$-dimensional stable manifold corresponding to the $N-1$ negative eigenvalues of $-L$ at the bifurcation and a one-dimensional unstable manifold corresponding to the bifurcating eigenvalue. The other two fixed points are both locally exponentially stable because they share the same $N-1$ negative eigenvalues as the origin and the bifurcating eigenvalue is also negative by \cite[Theorem~I.4.1]{Golubitsky1985}. Noticing that \eqref{eq:DynVector} is a positive monotone system and that all trajectories are bounded for $|u-1|$ sufficiently small, it follows from \cite[Theorem 0.1]{Hirsch1988} that almost all trajectories converge to the two stable equilibria, the stable manifold of the saddle separating the two basins of attractions.
The location of the three equilibria follows by direct substitution in the dynamic equations.

(iii) The first part of statement is just the definition of an $N$-parameter unfolding. The second part follows directly by \cite[Equation~I.3.23(d)]{Golubitsky1985} \hfill$\square$

In an all-to-all network and $ \bs \beta = \bs{0}$, the dynamics~\eqref{eq:DynVector} specialize to
\begin{equation}\label{eq:all-to-all-LTI}
\dot x_i = -(N-1) x_i+\sum_{\substack{ j=1, j\neq i}}^N u S(x_j),
\end{equation}
and Theorem~\ref{theorem:unfolding generic} holds globally in $u$ and $\bs x$.
\begin{corollary}
\label{theorem:all-to-all-unif}
The following statements hold for the stability of invariant sets of dynamics~\eqref{eq:all-to-all-LTI}: 
\begin{enumerate}
\item The consensus manifold is globally exponentially stable for each $u \in \mathbb{R}$, $u\geq 0$; 
\item $\bs{x}=\bs{0}$ is globally exponentially stable for $u \in[0,1)$ and globally asymptotically stable for $u=u^* =1$; 
\item  $\bs{x}=\bs{0}$ is exponentially unstable and there exist two locally exponentially stable equilibrium points $\pm y^s\bs{1}_N$ for $u>1$, where $y^s>0$ is the positive non-zero solution of $-y+uS(y)=0$. In particular, almost all trajectories converge to $\{y^s\bs{1}_N\}\cup\{-y^s\bs{1}_N\}$ for $u>1$.
\end{enumerate}
\end{corollary}
\begin{proof}
To prove (i) consider a Lyapunov function $V_{ij}(\bs x) = \frac{(x_i -x_j)^2}{2}$. It follows that
\begin{align*}
\dot V_{ij} (\bs x) &=-(N-1) (x_i - x_j) (x_i -x_j + u (S(x_i) -S(x_j))) \\
& < -(N-1) (x_i - x_j)^2 = -2(N-1) V_{ij}, 
\end{align*}
for all $x_i \ne x_j$. Therefore, for $V(\bs x) = \sum_{i=1}^n \sum_{j=1}^n V_{ij}(\bs x)$,
$
\dot V (\bs x) < -2(N-1) V (\bs x),
$
for all $\bs x \ne \zeta \bs 1_N$,  $\zeta \in \mathbb{R}$. $\dot{V}(\bs x) =  0$ for $x_i = x_j = \zeta$, so by LaSalle's invariance principle, the consensus manifold is globally exponentially stable. 

Using (i), it suffices to study dynamics~\eqref{eq:all-to-all-LTI} on the consensus manifold, where they reduce to the scalar dynamics 
\[
\dot y = -(N-1) y + u(N-1) S(y).
\]
(ii) and (iii) follow by inspection of these scalar dynamics and properties of $S$. 
\end{proof}

\subsection{Model extensions and further possible behaviors}\label{ssec:modelext}

Theorem~\ref{theorem:unfolding generic} shows that in the case of $\bs \beta = \bs 0$, there is a pitchfork bifurcation that results from the $Z_2$ symmetry in the dynamics \eqref{eq:DynVector}.  But even in the case $\bs \beta \neq \bs 0$, there can be $Z_2$ symmetry and thus a symmetric bifurcation.
One example is an all-to-all graph and two equally sized informed groups with $\beta_\text{A} = \beta_\text{B}$. If the size of each of the two informed groups is $n$, $2n\leq N$, then the vector field $F(\bs x, u, \bs \beta)$ commutes with the action of the nontrivial element of $Z_2$ represented by the matrix
$$\gamma=\left[\begin{array}{ccc}
0 & -I_{n} & 0\\
-I_{n} & 0 & 0\\
0 & 0 & -I_{N-2n}
\end{array}\right],\quad \gamma^2=I_{N},$$
where the zero blocks have suitable dimensions (see \cite[Lecture~1]{Fulton2013} for basic definitions and concepts from group representation theory). As in the proof of Theorem~\ref{theorem:unfolding generic}, this symmetry implies the presence of a $Z_2$-symmetric singularity, which for small $\bs\beta$ will again be a pitchfork. We thoroughly analyze a $Z_2$-symmetric informed network in Section~\ref{SSEC:Z2 symm pitch}, using model reduction to a low-dimensional invariant manifold. The general case will be provided in~\cite{Franci2018}.

From the informed $Z_2$-symmetric dynamics \eqref{eq:DynVector}, we can expect another significant behavior: the transition from a supercritical pitchfork to a subcritical pitchfork and emergence of two stable branches due to a stabilizing quintic term and two saddle-node bifurcations. In this case, the cubic term is a $Z_2$-symmetric unfolding term of the quintic pitchfork that modulates it between the standard cubic pitchfork and the subcritical cubic pitchfork. Figure~\ref{FIG:trans_to_quintic} shows this transition for increasing $\beta_A=\beta_B = \beta$ in a $Z_2$-symmetric graph. From a behavior perspective, the interesting fact about the subcritical pitchfork is the appearance of a bistable region between deadlock and decision, which will induce further robustness properties on the resulting decision-making. A rigorous analysis of this transition and its relevance in control problems will also be part of future works. 
 
\begin{figure}
\centering
\includegraphics[width=\linewidth]{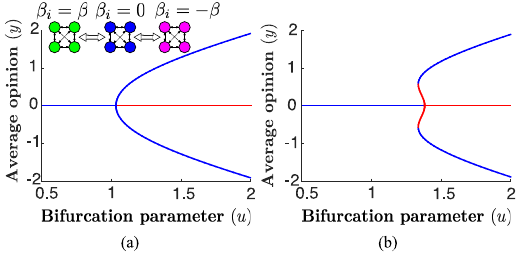}
\caption{Bifurcation in the $Z_2$-symmetric graph shown transitions from supercritical to subcritical with two saddle node bifurcations as $\beta$ increases. (a) $\beta = 1$. (b) $\beta = 3$. White arrows represent all-to-all undirected connections between groups.}\label{FIG:trans_to_quintic}
\end{figure}

Finally, we consider an extension of model \eqref{eq:DynVector} that will be important in the adaptive bifurcation control setting analyzed in Section~\ref{sec:feedback}. We let the social effort parameter $u$ be heterogeneous across the agents by considering the decision dynamics
\begin{equation} \label{eq:DynVector hetero us}
\boldsymbol{\dot{x}} = -D\boldsymbol{x} + UA\boldsymbol{S}(\bs{x}) + \bs{\beta},
\end{equation}
where $U={\rm diag}\big(\bar u+\tilde u_1,\ldots,\bar u+\tilde u_n\big)$ and $\sum_{i=1}^N\tilde u_i=0$. The value $\bar u$ is the average social effort and the $\tilde u_i$'s are the social effort heterogeneities. The evolution of the opinion of agent $i$ is governed by the dynamics
$$\dot x_i=-d_ix_i+\sum_{j=1}^N (\bar u+\tilde u_i) a_{ij}S(x_j) + \beta_i.$$
Let $\tilde{\bs u}=(\tilde u_1, \ldots, \tilde u_N)^T$ be the vector of social effort heterogeneities. The following theorem shows that the same results as in Theorem~\ref{theorem:unfolding generic} qualitatively persist for small heterogeneities in agent social efforts.

\begin{theorem}\label{theorem:unfolding generic hetero us}
The following hold for dynamics~\eqref{eq:DynVector hetero us} with fixed, strongly connected graph and sufficiently small $\tilde u_i$, $i\in\{1,\ldots,N\}$:
\begin{enumerate}
\item There exists a smooth function $\bar u^*(\tilde{\bs u})$ satisfying $\bar u^*(\bs 0)=1$ such that the linearization of~\eqref{eq:DynVector hetero us} for $\bs \beta=\bs 0$ possesses a unique zero eigenvalue at $(\bs x,\bar u)=(\bs 0,\bar u^*(\tilde{\bs u}))$. Moreover, the associated null right eigenvector $\tilde{\bs 1}$ satisfies $\|\tilde{\bs 1}-\bs 1\|_1=O(\|\tilde{\bs u}\|_1)$ and the associated singularity is isolated.
\item Let $g(y,\bar u,\bs 0)$ be the Lyapunov-Schmidt reduction of \eqref{eq:DynVector hetero us} with $\bs\beta=\bs 0$ at $(\bs x,u)=(\bs 0,\bar u^*(\tilde{\bs u}))$. The bifurcation problem $g(y,\bar u,\bs 0)$ has a symmetric pitchfork singularity at $(y, \bar u) =(0, \bar u^*(\tilde{\bs u}))$.
\item For $\bs\beta\neq \bs 0$, the bifurcation problem $g(y, \bar u, \bs\beta )$ is an $N$-parameter unfolding of the symmetric pitchfork.
\end{enumerate}
\end{theorem}
\begin{proof}
Let $\tilde F(\bs x,\bar u,\tilde{\bs u},\bs 0)$ denote the right hand side of \eqref{eq:DynVector hetero us} for $\bs\beta=\bs 0$. Observe that $\tilde F(\bs 0,\bar u,\tilde{\bs u},\bs 0)=\bs 0$. We show that there exists a smooth function $\bar u^*(\tilde{\bs u})$ with $\bar u^*(\bs 0)=1$ such that the Jacobian $J(\bar u,\tilde{\bs u})=\frac{\partial F}{\partial\bs x}(\bs 0,\bar u,\tilde{\bs u},\bs 0)$ is singular for $\bar u=\bar u^*(\tilde{\bs u})$ and sufficiently small $\tilde{\bs u}$. Moreover, there exist no other singular points close to $(\bs 0,\bar u^*(\tilde{\bs u}),\tilde{\bs u})$. To show this, we apply the implicit function theorem~\cite[Appendix 1]{Golubitsky1985} to the scalar equation
$$\det\left( J(\bar u,\tilde{\bs u}) \right)=0.$$
Using Jacobi's formula for the derivative of the determinant of a matrix, we obtain
$$\frac{\partial}{\partial \bar u}\det J(\bar u,\tilde{\bs u})=\mathrm{tr}\left(\mathrm{adj}(J)\frac{\partial J}{\partial \bar u}\right),$$
where $\mathrm{adj}(J)$ is the adjugate matrix of $J$~\cite{strang2006linear}. Because $J\mathrm{adj}(J)=\mathrm{adj}(J)J=\det(J)I_N$ and $\det(J(1,\bs 0))=\det L=0$, it follows that at $(\bar u,\tilde{\bs u})=(1,\bs 0)$ the image of $\mathrm{adj}(J)$ is the kernel of $J$ and that the image of $J$ is in the kernel of $\mathrm{adj}(J)$. Recalling that ${\rm rank}\,\mathrm{adj}(J)=N-{\rm rank}\,J=1$, it follows that $\mathrm{adj}(J(1,\bs 0))=c\bs 1_N\bs v_0^T$, where $\bs v_0^T$ is a left null eigenvector of $L$ and $c\neq 0$. Now, at $(\bar u,\tilde{\bs u})=(1,\bs 0)$,
$$\frac{\partial J}{\partial \bar u}=A.$$
$A$ is non-negative and, by the strong connectivity assumption, at least one element in each of its columns is different from zero. Furthermore, $\mathrm{adj}(J)\frac{\partial J}{\partial \bar u} = c\bs 1_N\bs v_0^T A$, and it follows that  $\mathrm{tr}(\mathrm{adj}(J)\frac{\partial J}{\partial \bar u}) =c \bs v_0^T A \bs 1_N$, which is non-zero. Consequently, $\frac{\partial}{\partial \bar u}\det J(\bar u,\tilde{\bs u}) \ne 0$. The existence of the smooth function $\bar u^*(\bs \beta)$ with the properties of the statement now follows directly from the implicit function theorem.

Using continuity arguments and the odd symmetry of \eqref{eq:DynVector hetero us}, the rest of the theorem statement follows by Theorem~\ref{theorem:unfolding generic}.
\end{proof}

\section{Model reduction and value sensitivity in agent-based model}\label{sec:ModelRed}
For special classes of graphs the dynamics \eqref{eq:DynVector} can be reduced to a low-dimensional manifold, which aids quantitative and numerical analysis.  We prove convergence of a reduced-order model in Section~\ref{subsec:ModelRed}. In Section~\ref{SSEC:Z2 symm pitch} we use the reduced model to prove the existence of the pitchfork bifurcation in informed $Z_2$-symmetric  all-to-all networks.  In Section~\ref{subsec:recover-val} we use the reduced model to show directly that our agent-based model~(\ref{eq:DynVectoru1u2}) recovers the value-sensitivity of the biological mean-field model discussed in Section~\ref{sec:HoneyBees}. This allows us to examine sensitivity of performance to parameters, such as group size and distribution of information across the group.

\subsection{Model reduction to low-dimensional, attracting manifold}\label{subsec:ModelRed}
For certain classes of network graph it is possible to identify a globally attracting, low-dimensional manifold to which the dynamics \eqref{eq:DynVector} converge, and to perform analysis on the reduced model. The dimension $N$ of the system is treated as a discrete parameter, allowing for the study of the sensitivity of the dynamics to the sizes of the informed and uninformed populations. As in \cite{Leonard2012}, where the decision-making behavior of animal groups on the move is considered, simulations of~\eqref{eq:DynVector}  show that under the conditions described below, the dynamics exhibit fast and slow timescale behavior. Initially agents with the same preference and neighbors reach agreement, and then in the slow timescale the dynamics of these groups evolve (see Figure~\ref{fig:Reduction}).

Let $n_1$ and $n_2$ be the number of agents with information $\beta_i = \beta_\text{A} = \bar{\beta}_1$ and $\beta_i = -\beta_\text{B} = \bar{\beta}_2$, respectively, and let $n_3 = N - n_1 - n_2$ be the number of agents with no information ($\beta_i = 0 = \bar{\beta}_3$). Let $\mathcal{I}_k \subset \{1, ..., N\}$, $k \in \{1,2,3\}$, be the index set associated with each group. Also assume $$a_{ij} = \begin{cases} \bar{a}_{km}, & \text{if} \ i \in \mathcal{I}_k, \ j \in \mathcal{I}_m, \ \text{and} \ i \neq j, \\ 0, & \text{otherwise}, \end{cases}$$ 
for $i,j \in \{1,...,N\}$, where $\bar{a}_{km} = 1$ if $k = m$. Under these assumptions, each node in the same group $k$ has the same in-degree, $\bar{d}_k= (n_k-1) + \sum_{m \ne k} \bar n_m a_{km}$, where $n_k$ is the cardinality of $\mathcal I_k$, and dynamics \eqref{eq:DynVector} for agent $i\in \mathcal{I}_k$ are
\begin{equation}
\label{eq:red-dyn1}
\dot{x}_i = -\bar{d}_k x_i + u \sum_{\substack{j \in \mathcal{I}_k \\ j \neq i}} S(x_j) + u\! \!\sum_{\substack{m \in \{1,2,3\} \\ m \neq k}} \sum_{j  \in \mathcal{I}_m}\bar{a}_{km}S(x_j) + \bar{\beta}_k.
\end{equation}
Theorem~\ref{theorem:all-to-all-reduction} allows the analysis of \eqref{eq:red-dyn1} to be restricted to the subspace where each agent in the same group has the same opinion. The  theorem is illustrated in the inset of Figure~\ref{fig:Reduction}.

\begin{figure}
\begin{center}
\includegraphics[width=63mm]{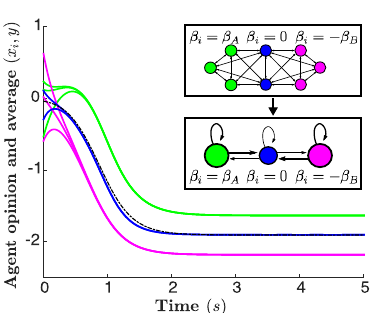}
\caption{Model reduction. Opinions are simulated over time for $N=8$ agents in the undirected network in the top box. $\beta_A = \beta_B = 1$ and $u = 2$. Opinions of agents of the same group (color) aggregate, and then the three group opinions evolve according to the reduced model shown in the bottom box. The black dashed-dotted line is the average group opinion.} 
\label{fig:Reduction}
\end{center}
\end{figure}

\begin{theorem}
\label{theorem:all-to-all-reduction}
Every trajectory of the dynamics~\eqref{eq:red-dyn1} converges exponentially to the three-dimensional manifold
\[
\mathcal{E} = \{\bs{x} \in \mathbb{R}^N | x_i = x_j, \ \forall i,j \in \mathcal{I}_k,\; k = 1,2,3\}.
\]
Define the reduced state as $\bs y = (y_1, y_2, y_3) \in \mathcal{E}$. Then, dynamics on $\mathcal{E}$ are
\begin{align} \nonumber
\dot{y}_1 = -\bar d_1 y_1 &+ u\big((n_1-1)S(y_1) + n_2\bar a_{12} S(y_2)  \\ & + n_3\bar a_{13} S(y_3)\big) + \beta_A \nonumber\\\label{eq:reduced-model}
\dot{y}_2 = -\bar d_2 y_2 &+ u\big((n_2-1) S (y_2) +  n_1\bar a_{21} S (y_1) \\ & +n_3\bar a_{23} S(y_3)\big) - \beta_B \nonumber \\ \nonumber
\dot{y}_3 = -\bar d_3 y_3 &+u\big((n_3-1)S(y_3) + n_1\bar a_{31} S(y_1)   \nonumber \\ &  + n_2\bar a_{32} S(y_2)\big). \nonumber
\end{align}
\end{theorem}
\begin{proof}
Let $V(\bs x) = \sum_{k = 1}^3 V_k(\bs x)$, where $V_k(\bs x) = \frac{1}{2}	\sum_{i \in \mathcal{I}_k} \sum_{j \in \mathcal{I}_k} (x_i - x_j)^2,\ \text{for}\ k \in \{1,2,3\}.$
It follows that
\begin{align*}
\dot{V}_k(\bs x) &= \sum_{i \in \mathcal{I}_k} \sum_{j \in \mathcal{I}_k} (x_i - x_j)(\dot x_i - \dot x_j) \\
&\!\!\! \!\!\! \!\!\!\!\!\! = \sum_{i \in \mathcal{I}_k} \sum_{j \in \mathcal{I}_k} \! \!\! \big( \! \!- \bar d_k(x_i - x_j)^2 - u(x_i - x_j)(S(x_i) - S(x_j))\big) \\
&\!\!\! \!\!\! \!\!\!\!\!\!  \leq - \bar d_k V_k (\bs x),
\end{align*}
so $\dot V(\bs x) \leq - \bar d_kV(\bs x)$. By LaSalle's invariance principle, every trajectory of \eqref{eq:red-dyn1} converges exponentially to the largest invariant set in $V(\bs x) = 0$, which is $\mathcal{E}$. Let $y_k = x_i$, for any $i \in \mathcal{I}_k, \ k \in \{1,2,3\}$.  Then dynamics \eqref{eq:red-dyn1} reduce  to \eqref{eq:reduced-model}.  
\end{proof}

\subsection{Pitchfork bifurcation in informed $Z_2$-symmetric all-to-all networks}
\label{SSEC:Z2 symm pitch}
Dynamics \eqref{eq:DynVector} are $Z_2$-symmetric if $\beta_A = \beta_B = \beta$, $n_1 = n_2 = n$, $2n \leq N$, and the graph is symmetric with respect to the two informed groups and their influence on the uninformed group. This symmetry is satisfied for the class of graph discussed in Section~\ref{subsec:ModelRed} if  $\bar a_{km} = \bar a_{mk}$, for each $k, m \in \{1,2,3\}$, and $\bar a_{13} = \bar a_{23}$. For this class, $Z_2$ symmetry means that reversing the sign of $\beta_A$ and $\beta_B$ is equivalent to applying the transformation $\bs x \mapsto \bs{-x}$.

Consider an all-to-all graph with unit weights and $\beta_A = \beta_B = \beta$, $n_1 = n_2 = n$, $n_3 - 2n \geq 0$, which make the dynamics \eqref{eq:DynVector} $Z_2$-symmetric as discussed in Section~\ref{ssec:modelext}. We can find an approximation $\hat u^*$ to the bifurcation point $u^*$ by examining the reduced dynamics \eqref{eq:reduced-model}, which specialize to
\begin{align} \nonumber
\dot{y}_1 = -(N-1)y_1 &+u\big((n-1)S(y_1)\\ \nonumber &+n S(y_2)+n_3 S(y_3)\big) + \beta\\ \label{eq:reduced-model-ata}
\dot{y}_2 =  -(N-1)y_2 &+ u\big( nS (y_1)\\ \nonumber &+(n-1) S (y_2)+n_3 S(y_3)\big) - \beta\\ \nonumber
\dot{y}_3=  -(N-1)y_3 &+u\big( n S(y_1)\\ \nonumber &+n S(y_2)+(n_3-1)S(y_3)\big).
\end{align}
Because of $Z_2$ symmetry, the deadlock state $\bs{y}^*(u,\beta) = (y^*(u,\beta),-y^*(u,\beta),0)$ is always an equilibrium, where $y^*(u,\beta)$ is the solution to 
\begin{equation}
\label{eq:ystarred}
(N-1)y^* + uS(y^*) - \beta = 0.
\end{equation}
When $\beta = 0$, $y^*(u,0)=0$ for all $u\in\mathbb R$. When $\beta\neq0$, the implicit function theorem ensures that $y^*(u,\beta)$ depends smoothly on $u$ and $\beta$.
By Taylor expansion, an approximation to ${y}^*(u,\beta)$ can be found, and the bifurcation point where deadlock becomes unstable can also be approximated. To compare theoretical and numerical results, we let $S(\cdot)=\tanh(\cdot)$ in Theorem~\ref{thm:bif-pt} but the computations are general.

\begin{theorem}\label{thm:bif-pt}
The following hold for dynamics~\eqref{eq:reduced-model-ata} with $S(\cdot) =\tanh(\cdot)$:  
\begin{enumerate}
\item the equilibrium  $\bs{y}^*(u,\beta) = (y^*(u,\beta),-y^*(u,\beta),0)$ satisfies
\begin{equation}\label{eq:eqm-pt}
y^*(u,\beta) = \frac{1}{N-1+u}\beta + \frac{u}{3(N-1+u)^4}\beta^3 + \mathcal{O}(\beta^5);
\end{equation} 
\item the equilibrium $\bs{y}^* = (y^*,-y^*,0)$ is singular for 
\begin{equation}
{u}^* = 1 + \frac{(1+3N^3)^2(N-n_3)}{9N^9}\beta^2 + \mathcal{O}(\beta^4);
\label{eq:ustar-approx}
\end{equation}
\item for sufficiently small $\beta$, the singularity at $u=u^*$ is a pitchfork.
\end{enumerate}
\end{theorem}
\begin{proof}
We  begin with (i). Consider the Taylor series expansion of $y^*(u,\beta)$ with respect to $\beta$:
\begin{equation}\label{eq:eqm-taylor}
y^*(u,\beta) = \beta y_I + \beta^2 y_{II} + \beta^3 y_{III} + \beta^4 y_{IV} + \mathcal{O}(\beta^5).
\end{equation}
Substitute~\eqref{eq:eqm-taylor} for ${y}^*(u,\beta)$ into \eqref{eq:ystarred} and differentiate with respect to $\beta$ to get 
\[
(N-1){y^{*}}'(u,\beta) + u\text{sech}^2\big(y^*(u,\beta)\big){y^*}'(u,\beta) - 1 = 0.
\]
Letting $\beta = 0$ yields $y_I = \frac{1}{N-1+u}$. Proceeding similarly for higher order derivatives gives $y_{II} = y_{IV} = 0$ and
$y_{III} = \frac{u}{3(N-1+u)^4}.$ Substituting these values into~\eqref{eq:eqm-taylor} yields~\eqref{eq:eqm-pt}, establishing (i). 

At a singular point $u^*$, the Jacobian of \eqref{eq:reduced-model-ata} computed at $\bs y^*$ drops rank. The Jacobian of~\eqref{eq:reduced-model-ata} at $\bs y^*$ is
\[
\left[\begin{smallmatrix}
-(N-1) + u(n-1) S'(y^*) & u n S'(y^*) & u n_3 \\
un S'(y^*) & -(N-1) + u(n-1) S'(y^*) & u n_3 \\
u n S'(y^*) & u n S'(y^*) & -(N-1)+ u(n_3-1)
\end{smallmatrix}\right],
\]
where we have used the fact that $S'(\cdot)$ is an even function. 
For $S(\cdot) = \text{tanh}(\cdot)$ the determinant $d$ of the Jacobian is
\begin{multline*}
\!\!\! d =   -\frac{1}{4}\eta(-1 + N + 2u + \eta \text{cosh}(2y_1))(\eta + 3u + n_3u - 2Nu \\ - 2u^2 
+ (\eta + u - n_3 u)\text{cosh}(2y_1))\text{sech}^4(y_1), 
\end{multline*}
with $\eta = N-1$. A positive $u = u^*$ for which $d=0$ satisfies
\begin{multline*}
u^* = \frac{1}{4}\Big(3 + n_3 - 2N + \text{cosh}(2y^*) - n_3\text{cosh}(2y^*) \\
+ \sqrt{16\eta \text{cosh}^2(y^*) + (3 + n_3 - 2N - (-1 + n_3)\text{cosh}^2(2y^*))}\Big).
\end{multline*}
Note that $y^*$ is also a function of $u^*$ and the above equation is a transcendental  equation in $u^*$, which can be solved numerically. To compute $u^*$ we use the Taylor Series expansion 
$u^*(\beta) = 1 + u_I^* \beta + u_{II}^* \beta^2 + u_{III}^* \beta^3 + O(\beta^4)$ and match coefficients to 
obtain~\eqref{eq:ustar-approx}.

To prove that the singular point $(\bs y^*,u^*,\beta)$ corresponds to a pitchfork we invoke singularity theory for $Z_2$-symmetric bifurcation problems~\cite[Chapter VI]{Golubitsky1985}. By Theorem~\ref{theorem:unfolding generic}, the singular point $(\bs y^*,u^*,\beta)$ is a pitchfork for $\beta=0$. Because~\eqref{eq:reduced-model-ata} is $Z_2$-symmetric, for sufficiently small $\beta>0$ we obtain a small $Z_2$-symmetric perturbation of the pitchfork at $\beta=0$. Invoking genericity of the pitchfork in $Z_2$-symmetric systems \cite[Theorem VI.5.1]{Golubitsky1985}, we conclude that \eqref{eq:reduced-model-ata} possesses a pitchfork at $u=u^*$.
\end{proof}

\begin{remark}[Deadlock breaking in informed symmetric case] \textup{Theorem~\ref{thm:bif-pt} shows persistence of the bifurcation in Corollary~\ref{theorem:all-to-all-unif} under $Z_2$ symmetry and weakly informed agents ($\beta$ small). For small social effort ($u<u^*$), the undecided state, and therefore deadlock, is the only stable equilibrium despite the presence of informed agents, but for sufficiently large social effort ($u>u^*$) the undecided state is unstable. For $\beta \neq 0$ this does not  imply that deadlock has been broken, because differences in opinions can maintain disagreement ($\delta \neq 0$) for moderate social effort. The adaptive bifurcation dynamics introduced in Section~\ref{sec:feedback} overcomes this  by implementing feedback control that increases the social effort beyond the destabilizing of the undecided state to yield agreement as well and thus deadlock breaking.}
\end{remark}

\begin{remark}[Influence of system parameters] \textup{The approximation  \eqref{eq:ustar-approx} of  $u^*$ depends on information strength $\beta$, total group size $N$ and uninformed group size $n_3$. It thus explicitly describes the sensitivity of the bifurcation to group size, information strength and distribution of information. We  further illustrate this influence in Section~\ref{subsec:influence}.}
\end{remark}

\subsection{Value-sensitive decision-making in $Z_2$-symmetric networks}\label{sec:Results}
Returning to the timescale of \eqref{eq:DynVectoru1u2}, we show how the agent-based dynamics \eqref{eq:DynVectoru1u2} recover the value-sensitivity and performance of the honeybee decision-making dynamics studied with the mean-field model in \cite{Pais2013}, and discussed in Section~\ref{sec:HoneyBees}. We then use the agent-based model to examine the influence on performance of system parameters, including size of the group and strength and distribution of information across the group.

\subsubsection{Value-sensitivity and performance in agent-based model}\label{subsec:recover-val}
To show  value-sensitivity and the associated robustness and adaptability in the agent-based model, we examine the dynamics \eqref{eq:DynVectoru1u2} with alternatives of equal value $\nu_A=\nu_B = \nu$, inertia $u_I = 1/\nu$, and bifurcation parameter $u_S = u \cdot u_I =\frac{u}{\nu}$. Applying \eqref{eq:ustar-approx} gives the approximation $\hat u_S^*$ to the bifurcation point for  \eqref{eq:DynVectoru1u2} as $u_S^* = \hat u_S^* + \mathcal{O}(\nu^7)$, where 
\begin{equation} \label{eq:muhat}
\hat{u}_S^* = \frac{1}{\nu} + \frac{(1 +3N^3)^2(N-n_3)}{9N^9}\nu^3.
\end{equation}
Figure~\ref{fig:AnalysisResults}(a) shows how well  $\hat u_S^*$ approximates $u_S^*$ computed using MatCont continuation software \cite{Govaerts:2015aa}.  As in the case of the honeybee mean-field model, the bifurcation point in the agent-based model depends inversely on the value of the alternatives $\nu$ (see \eqref{eq:muhat} and Figure~\ref{fig:AnalysisResults}(a)).  Thus,  our agent-based decision-making model recovers the value-sensitive decision-making of the honeybee mean-field model.

This value-sensitivity  is demonstrated further in Figure~\ref{fig:AnalysisResults}(b), where bifurcation diagrams for the agent-based model are given for a range of values $\nu$.  As $\nu$ is increased, the bifurcation point decreases and  the sharpness of the bifurcation branches increases, representing a faster increase in average opinion.

\subsubsection{Influence of system parameters in agent-based model} \label{subsec:influence}
An advantage of the agent-based framework is that it makes it possible to systematically study sensitivity of the dynamics to model parameters including those that describe network structure and heterogeneity.
An examination of \eqref{eq:muhat} shows that $\hat{u}_S^*$ decreases with increasing total group size $N$, implying that less social effort is required to make a decision with a larger group. In the limit as $N$ increases, $\hat u^*_S = \frac{1}{\nu}$.

Figure~\ref{fig:mustar-with-nU} shows the inverse relationship between $\nu$ and $\hat{u}_S^*$ for different values of $n_3$ with fixed $N$ and $n_1=n_2 = \frac{N-n_3}{2}$.   As $n_3$ increases, the number of informed agents decreases, and the curve drops, implying that increasing the number of uninformed agents reduces the requirement on social effort to destabilize deadlock.   This result suggests that the agent-based dynamics could be mapped to describe the schooling fish decision dynamics discussed in \cite{Couzin2011}, where it is shown that increasing the number of uninformed agents allows the group to choose the alternative preferred by the majority over a more strongly influencing but minority preference.

\begin{figure}
	\centering
	\includegraphics[width=\linewidth]{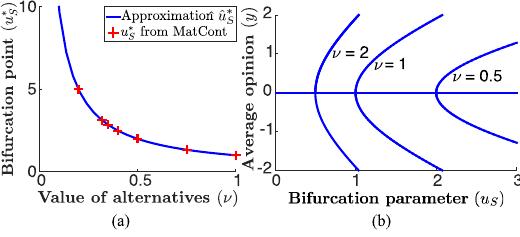}
\caption{Value sensitivity in the agent-based model~\eqref{eq:reduced-model-ata} for alternatives with equal value $\nu$ and informed and uninformed group sizes $n_1=n_2= 10$ and $n_3 = 80$. (a) The blue line shows the approximation $\hat u_S^*$ \eqref{eq:muhat} while the red crosses show $u_S^*$ computed numerically using MatCont continuation software. (b) Bifurcation diagrams for three values of $\nu$.}
	\label{fig:AnalysisResults}
\end{figure}

	\begin{figure}[t]
	\centering
    \includegraphics[width=50mm]{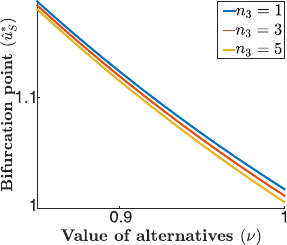}
    \caption{The inverse relationship between $\hat u_S^*$ from \eqref{eq:muhat} and the equal value of the alternatives $\nu$ in the agent-based model for three values of  $n_3$ with $N=7$ and $n_1=n_2 = \frac{N - n_3}{2}$.}
        \label{fig:mustar-with-nU}
	\end{figure}

\section{Adaptive control of bifurcation}\label{sec:feedback}

Studying the decision-making process as a control system allows us to add further layers of dynamics beyond those considered in \cite{Seeley12} and \cite{Pais2014}. A natural extension to our model is to regard the social effort parameter $u$ as a control input, and to give it an adaptive feedback dynamic. We have shown that the level of social effort required to break deadlock depends on system parameters that may not be known ahead of time, so it can be highly useful to allow the system to adapt to its circumstances.

We introduce to the agent-based model a slow feedback dynamic that drives the magnitude of opinion of the group $|y|$ to a prescribed threshold value $y_{th}>0$, thereby ensuring a decision is made. As the system is decentralized, each agent updates its own social effort control $u_i$.  Thus, we use the extended model \eqref{eq:DynVector hetero us}. Moreover, we do not make the assumption that agents can directly access the average opinion $y$.

The proposed adaptive controller consists of two phases. We first let each agent estimate the group average $y$ using the finite-time dynamic consensus algorithm proposed in \cite{George:2017aa}:
\begin{subequations}
\begin{align}
\dot{\boldsymbol{w}} & = -\alpha\text{sgn}(L\hat{\boldsymbol{y}}) \label{eq:con_alg_1}\\
\hat{\boldsymbol{y}} & = L\boldsymbol{w}+\boldsymbol{x}, \label{eq:con_alg_2}
\end{align}
\end{subequations}
where $\hat{\bs y}$ is the vector of agent estimates of $y= \sum_{i=1}^Nx_i$, $w_i$ are auxiliary variables, and $\alpha>0$ is the estimator gain. During this phase, $\dot u_i=\dot x_i=0$ for all $i$. It is shown in \cite[Theorem 1]{George:2017aa} that the consensus algorithm given by \eqref{eq:con_alg_1}-\eqref{eq:con_alg_2} guarantees that the error $\tilde{\boldsymbol{y}} = \hat{\boldsymbol{y}} - \frac{1}{N}\boldsymbol{1}_N^T\boldsymbol{x}\boldsymbol{1}_N=\hat{\bs y}-y\bs 1_N$ is globally finite-time convergent to zero. The convergence time $s^*$ is explicitly given by
$$s^* \le  \frac{||\tilde{\boldsymbol{y}}(s_0)||}{\lambda_2(L)},$$
where $\lambda_2(L)$ is the second smallest eigenvalue of $L$. Therefore $\hat y_i(s)=y$ for all $s\geq s^*$. Here, we assume that each agent can compute a lower bound on $\lambda_2(L)$. This can be accomplished distributedly using algorithms developed in~\cite{aragues2014distributed}.

For $s> s^*$, we let $x_i$ and $u_i$, $i \in \{1, \ldots, N\}$, evolve according to the two-time scale adaptive dynamics
\begin{IEEEeqnarray*}{rCl}
\dot{\bs x}&=&-D\boldsymbol{x} + UAS(\boldsymbol{x}) +\bs{\beta},\\
\dot{\bs u}&=&\epsilon\left(y_{th}^2\boldsymbol{1}_N - \hat{\boldsymbol{y}}^2\right),\\
&=&\epsilon\left(y_{th}^2- y^2\right)\boldsymbol{1}_N,
\end{IEEEeqnarray*}
where $U={\rm diag}(u_1,\ldots,u_N)$ and $\hat{\bs{y}}^2$ is the vector of the squares of the elements of $\hat{\bs y}$. We omit specifying $s\geq s^*$ from now on. Let $\bar u = \frac{1}{N}\sum_{i=1}^N u_i$, then for $s >s^*$, $\dot{\bar u} = \dot u_i$, for each $i$. Thus, the social effort differences $\tilde u_i=u_i-\bar u$ are constant and the adaptive dynamics reduce to
\begin{IEEEeqnarray}{rCl}\label{aftersstarx}
\IEEEyesnumber
\dot{\bs x}&=&-D\boldsymbol{x} + UAS(\boldsymbol{x}) +\bs{\beta},\IEEEyessubnumber\\
\dot{\bar u}&=&\epsilon\left(y_{th}^2- y^2\right),\IEEEyessubnumber
\end{IEEEeqnarray}
where $U={\rm diag}(\bar u+\tilde u_1,\ldots,\bar u+\tilde u_N)$.

We study the behavior of \eqref{aftersstarx} using geometric singular perturbation theory \cite{Fenichel1979a}. We first find a suitable low-dimensional invariant manifold for \eqref{aftersstarx} by means of the center manifold theorem \cite[Theorem 3.2.1]{Guckenheimer2002}.
To use the center manifold computation we extend \eqref{aftersstarx} with dummy dynamics for $\epsilon$ and $\bs \beta$~\cite[\S 18.2]{wiggins2003introduction}:
\begin{IEEEeqnarray}{rCl}\label{aftersstarx_dummy}
\IEEEyesnumber
\dot{\bs x}&=&-D\boldsymbol{x} + UAS(\boldsymbol{x}) +\bs{\beta},\IEEEyessubnumber\label{aftersstarx_dummya}\\
\dot{\bar u}&=&\epsilon\left(y_{th}^2- y^2\right),\IEEEyessubnumber\\
\dot\epsilon&=&0,\IEEEyessubnumber\\
\dot{\bs\beta}&=& \bs 0.\IEEEyessubnumber
\end{IEEEeqnarray}

By Theorem~\ref{theorem:unfolding generic hetero us}, if the graph is strongly connected, the linearization of \eqref{aftersstarx_dummy} at $(\bs x, \bar u, \epsilon, \bs \beta) = (\bs 0,\bar u^*(\tilde{\bs u}),0,\bs 0)$ has $N-1$ eigenvalues with negative real part and $3+N$ zero eigenvalues, with corresponding null eigenvectors
\begin{equation*}
\bs e_{\tilde{\bs 1}}=\begin{bmatrix}\tilde{\bs 1} \\ 0 \\ 0 \\ \bs 0\end{bmatrix},\bs e_u= \begin{bmatrix} \bs{0} \\ 1 \\ 0 \\ \bs 0 \end{bmatrix}, \bs e_\varepsilon=\begin{bmatrix}\bs{0} \\ 0 \\ 1 \\ \bs 0 \end{bmatrix},\bs e_{\bs\beta,i}\begin{bmatrix}\bs{0} \\ 0 \\ 0 \\ \bs e_i \end{bmatrix},\ i=1,\ldots,N, 
\end{equation*}
where $\tilde{\bs 1}$ is defined in Theorem~\ref{theorem:unfolding generic hetero us} and $\bs e_i$ is the $i$-th vector of the standard basis of $\mathbb R^N$. It follows by the center manifold theorem, that \eqref{aftersstarx_dummy} possesses an $(N+3)$-dimensional center manifold $W^c={\rm span}\{\bs e_{\tilde{\bs 1}},\bs e_u,\bs e_\varepsilon,\bs e_{\bs\beta,i},\ i=1,\ldots,N\}$ that is exponentially attracting. Dropping the dummy dynamics, the restrictions of \eqref{aftersstarx_dummy} to $W^c$ are
\begin{IEEEeqnarray}{rCl}\label{aftersstarx_center}
\IEEEyesnumber
\dot y_c&=&g_c(y_c,\bar u,\bs\beta),\IEEEyessubnumber\label{aftersstarx_centera}\\
\dot{\bar u}&=&\epsilon\left(y_{th}^2- y^2\right),\IEEEyessubnumber
\end{IEEEeqnarray}
where $g_c$ is the reduction of the vector field~\eqref{aftersstarx_dummya} onto its center manifold. Similar to the Lyapunov-Schmidt reduction, the center manifold reduction also preserves symmetries of the vector-field~\cite[\S 1.3]{golubitsky2003symmetry}. It follows similarly to
Theorem~\ref{theorem:unfolding generic hetero us} that for $\bs\beta=\bs 0$, the reduced fast vector field $g_c(y_c,\bar u,\bs 0)$ possesses a $Z_2$-symmetric pitchfork singularity at $(y_c,\bar u)=(0,\bar u^*(\tilde {\bs u}))$, and $g_c(y_c,\bar u,\bs \beta)$ is an $N$-parameter unfolding of the pitchfork. Dynamics~\eqref{aftersstarx_center} capture the qualitative behavior of dynamics~\eqref{aftersstarx} for initial conditions sufficiently close to $(\bs x,\bar u)=(\bs 0,\bar u^*(\tilde{\bs u}))$ and small $\bs\beta$.

Behavior of equations \eqref{aftersstarx_center} can be analyzed using geometric singular perturbation theory \cite{Fenichel1979a,Krupa2001b,BerglundGentz2006}. We define the slow time $\tau=\epsilon s$, which transforms \eqref{aftersstarx_center} into the equivalent dynamics
\begin{align*}
\epsilon y' & = g_c(y,\bar{u},\bs\beta),\\
\bar{u}' & = \left(y_{th}^2 - y^2\right),
\label{EQ:Dyn3}
\numberthis
\end{align*}
where $'=\frac{d}{d\tau}$. In the singular limit $\epsilon \to 0$, the {\it boundary layer} dynamics evolving in fast time $s$ are
\begin{align*}
\dot y & = g_c(y,\bar{u},\bs\beta),\\
\dot{\bar{u}} & =0,
\label{EQ:Dyn2layer}
\numberthis
\end{align*}
and the {\it reduced} dynamics evolving in slow time $\tau$ are:
\begin{align*}
0 & = g_c(y,\bar{u},\bs\beta),\\
\bar{u}' & = \left(y_{th}^2 - y^2\right).
\label{EQ:Dyn3reduced}
\numberthis
\end{align*}
The slow dynamics are defined on the {\it critical manifold} $M_0=\{(y,\bar{u}):g_c(y,\bar{u},\bs\beta)=0\}$. 
Singular perturbation theory provides a qualitative picture of trajectories of the original dynamics; trajectories of the boundary layer dynamics are a good approximation of the original dynamics far from $M_0$, whereas trajectories of the reduced dynamics are a good approximation close to $M_0$.

Trajectories of the boundary layer and reduced dynamics are sketched in Figure~\ref{fig:Control_Simulations}(a), (c) and (f) for different qualitatively distinct cases. When $\bs\beta=0$ (Figure~\ref{fig:Control_Simulations}(a)), by Theorem~\ref{theorem:all-to-all-unif}, $M_0$ is composed of a single globally exponentially stable branch $y=0$ for $\bar{u} < \bar{u}^*$ and three branches emerging from a pitchfork bifurcation for $\bar{u}>\bar{u}^*$. The outer branches $y=\pm\bar y(\bar{u})$ are locally exponentially stable and  $y=0$ is unstable. For a given $\bar u$, trajectories of the boundary layer dynamics (double arrows) converge toward stable branches of $M_0$ and are repelled by unstable branches. On $M_0$, trajectories of the reduced dynamics (single arrow) satisfy $\bar{u}'>0$ if $|y|<y_{th}$ and $\bar{u}'<0$ if $|y|>y_{th}$.

\begin{figure*}
	\centering
    \includegraphics[width=120mm]{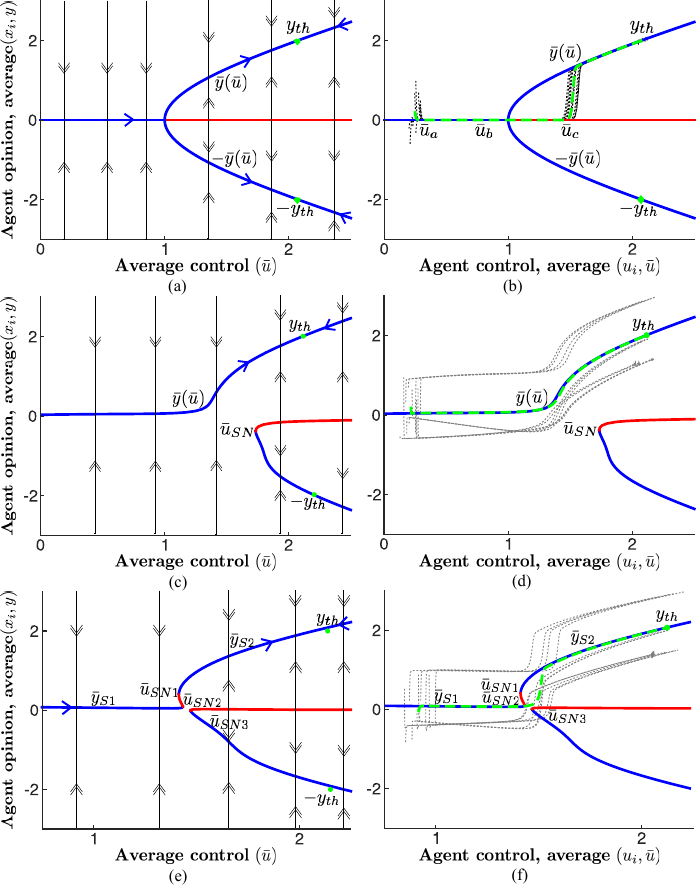}
    \caption{Illustration of the slow-fast controlled adaptive dynamics \eqref{aftersstarx}, for the graph shown in Figure~\ref{fig:univ-unfolding}. Shown in (a), (c), (e) are singular phase portraits of $y$ versus\ $\bar{u}$, where the double arrows illustrate the boundary layer dynamics and the single arrows illustrate the reduced dynamics. Shown in (b), (d), (f) are simulated trajectories of $x_i$ versus $u_i$ in dotted gray and $y$ versus $\bar u$ in dashed green.  (a) and (b) $\beta_A = \beta_B = 0$. (c) and (d) $\beta_A = \beta_B = -5$.  (e) and (f) $\beta_A = -5$ and $\beta_B = -4.5$.}
\label{fig:Control_Simulations}
\end{figure*}

When $\bs\beta\neq\bs 0$ there are four possible qualitatively distinct cases, corresponding to the four persistent bifurcation diagrams of the pitchfork illustrated in Figure~\ref{fig:univ-unfolding}(a). We consider the two bifurcation diagrams in the left column of Figure~\ref{fig:univ-unfolding}(a), as the other two are just their reflections. In the first case (Figure~\ref{fig:Control_Simulations}(c)), there exists a smooth attracting branch of the critical manifold that connects deadlock to decision. Trajectories slide along it until the threshold is reached.  In the second case (Figure~\ref{fig:Control_Simulations}(e)), the deadlock branch folds at $\bar u=\bar u_{SN2}$. At that point, we expect trajectories to jump to the stable decision branch and slide until the threshold is reached.

Theorem~\ref{THM:closed loop_1} summarizes the qualitative behavior of the adaptive control dynamics \eqref{aftersstarx_center} for sufficiently small $\epsilon$ and appropriate ranges of $\bar{u}(0)$. The results are illustrated with simulations in Figure~\ref{fig:Control_Simulations}(b), (d) and (f). The $\bar u_{SN}$'s denote the value of $\bar u$ at saddle-node bifurcation points in the perturbed bifurcation diagrams.

\begin{theorem}\label{THM:closed loop_1}
Suppose the interconnection graph is strongly connected. There exists $\bar\epsilon>0$ such that, for all $\epsilon\in(0,\bar\epsilon]$ and all $y_{th}>0$, the following hold for dynamics~\eqref{aftersstarx_center}.\\
$\bullet$ $\bs\beta=\bs 0$ (Figure~\ref{fig:Control_Simulations}(a)). For all initial conditions $y_c(0),\bar u (0)$ such that $|y_c(0)|$ is sufficiently small, $0<\bar u(0)<\bar u^*(\tilde{\bs u})$, and $|\bar u(0)-\bar u^*(\tilde{\bs u})|$ is sufficiently small, the trajectory of~\eqref{aftersstarx_center} exponentially approaches the critical manifold at a point $\bar{u}_a=\bar{u}(0)+\mathcal O(\epsilon|\log(\epsilon)|)$, stays close to it with $\dot{\bar u}>0$ until $u=\bar u_c=2\bar u^*(\tilde{\bs u})-\bar u_b$, where $\bar u(0)<\bar u_b<\bar u^*(\tilde{\bs u})$, and afterwards exponentially converges to the upper or lower attracting branches of the critical manifold, until $y \in \{\pm y_{th}\}$.\\
$\bullet$ $\bs\beta\neq\bs 0$, Case 1
(Figure~\ref{fig:Control_Simulations}(c)). For all initial conditions $y_c(0),\bar u (0)$ such that $|y_c(0)|$ is sufficiently small, $0<\bar u(0)$ and sufficiently small, the trajectory of~\eqref{aftersstarx_center} exponentially approaches the critical manifold at a point $\bar{u}(0)+\mathcal O(\epsilon|\log(\epsilon)|)$, and slides along the critical manifold with $\dot{\bar u}>0$ until $y=y_{th}$.\\
$\bullet$ $\bs\beta\neq\bs 0$, Case 2 
(Figure~\ref{fig:Control_Simulations}(e)). For all initial conditions $y_c(0),\bar u (0)$ such that $|y_c(0)|$ is sufficiently small, $0<\bar u(0)<\bar u_{SN1}$, and $|\bar u(0)-\bar u_{SN1}|$ is sufficiently small, the trajectory of~\eqref{aftersstarx_center} exponential approaches the critical manifold at a point $\bar{u}(0)+\mathcal O(\epsilon|\log(\epsilon)|)$, slides along the critical manifold with $\dot{\bar u}>0$ until it reaches the fold at $\bar u_{SN2}$, at which points it jumps to the upper attracting branch of the critical manifold and slides along it until $y=y_{th}$.
\end{theorem}
\begin{proof}
We just sketch the proof of Theorem~\ref{THM:closed loop_1}.\\
$\bullet$ $\bs\beta=\bs 0$. Noticing that the critical manifold is normally hyperbolic (see \cite[Section 1.2]{Jones1995a}) and attracting for $\bar u<\bar u^*(\tilde{\bs u})$, the existence of the point $\bar u_a$ follows by the standard Fenichel theory~\cite{Fenichel1979a},\cite[Theorems~1,3]{Jones1995a}. The existence of points $\bar u_b$ and $\bar u_c$ follows by \cite[Theorem 2.2.4]{BerglundGentz2006}. The rest of the statement follows again by standard Fenichel theory.\\
$\bullet$ $\bs\beta\neq \bs 0$, Case 1. Because the smooth branch of the critical manifold connecting deadlock and decision is normally hyperbolic and attracting, the statement follows directly by the standard Fenichel theory.\\ 
$\bullet$ $\bs\beta\neq \bs 0$, Case 2. Because the branch of the critical manifold for $\bar u<\bar u_{SN1}$ is normally hyperbolic and attracting, the existence of the point $\bar u_a$ follows by the standard Fenichel theory. The behavior of trajectories through the fold $\bar u_{SN2}$, follows by \cite[Theorem~2.1]{Krupa2001b}. The rest of the statement follows by the standard Fenichel theory.
\end{proof}

\begin{remark}[Bifurcation delay] \textup{The passage through the pitchfork for $\bs\beta=\bs 0$ is characterized by a ``bifurcation delay", that is, the system state lies close to the unstable branch of equilibria $y=0$ for an $\mathcal O(1)$ range of $\bar u$ after the pitchfork bifurcation (at $\bar u=\bar u_c$). This delay is illustrated in Figure~\ref{fig:Control_Simulations}(b). Intuitively, the delay is due to the fact that near the bifurcation point, dynamics~\eqref{aftersstarx_centera} slow down significantly and singular perturbation arguments based on the timescale separation do not hold.}
\label{rem:delay}
\end{remark}

\begin{remark}[Guaranteed deadlock breaking] For any $y_{th}>0$, the proposed feedback control ensures that deadlock is broken for sufficiently small $\bs\beta$.
\end{remark}

Note that since center manifold theory is a local theory, Theorem~\ref{THM:closed loop_1} captures the behavior of the full dynamics~\eqref{aftersstarx} only close to the singular point $(\bs x,\bar u,\bs\beta)=(\bs 0,\bar u^*(\tilde{\bs u}),\bs 0)$. However, since in the all-to-all case with zero information and zero social effort differences, the consensus manifold is {\it globally} exponentially attracting (and invariant), we have the following straightforward corollary of Theorem~\ref{THM:closed loop_1}.

\begin{corollary}[All-to-all system]
If the graph is all-to-all and the social effort differences are zero ($\tilde{\bs u}=\bs 0$), then the same result as Theorem~\ref{THM:closed loop_1} for $\bs\beta=\bs 0$ holds globally in $\bs x$ and $\bar u$.
\end{corollary}

We further stress that numerical simulations suggest that the results of Theorem~\ref{THM:closed loop_1} hold globally in the generic case of a strongly connected graph with non-zero information and non-zero social effort differences.

\section{Conclusion}
\label{sec:final}
We have defined an agent-based model that describes distributed dynamics for collective decision-making between alternatives in a multi-agent network with a generic strongly connected graph.  The agent-based dynamics are  carefully designed to exhibit a pitchfork bifurcation so that  mechanisms of collective animal behavior can be formally translated into bio-inspired control design for multi-agent decision-making. We have rigorously established the nonlinear phenomena associated with the proposed dynamics.
We have used this framework to prove agent-based dynamics that inherit the value sensitivity and robust and adaptive features of honeybee nest site selection.   The sensitivity of outcomes to model parameters and heterogeneity have been investigated. An adaptive control law has been designed for the bifurcation parameter to achieve unanimous decision-making in the network.  The framework will be used for control inspired by high performing decision-making of schooling fish and other animal groups.  Future work will also seek to characterize nonlinear behavior away from the bifurcation point $u\gg 1$, for larger values of $\bs \beta$, and when more than two options are available.

\begin{IEEEbiography}[{\includegraphics[width=1in,height=1.25in,clip,keepaspectratio]{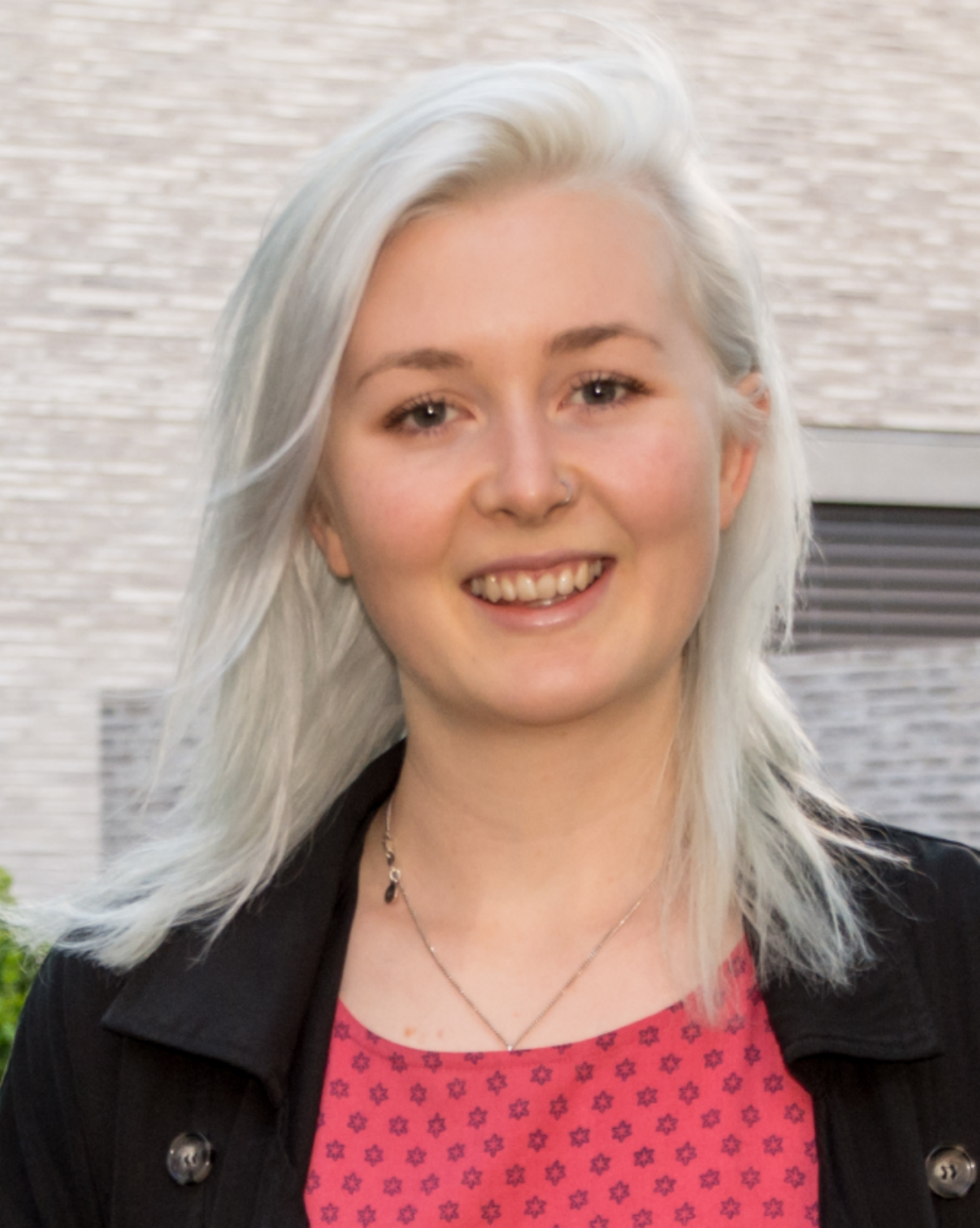}}]{Rebecca Gray} received the B.E.(Hons) degree (2013) in mechanical engineering from the University of Canterbury, Christchurch, Auckland and the M.A. degree in mechanical engineering from Princeton University, Princeton, NJ. She is a graduate student in the Mechanical and Aerospace Engineering Department at Princeton University, and her research is in modelling and control of bio-inspired collective decision-making.
 \end{IEEEbiography}
 
 \begin{IEEEbiography}[{\includegraphics[width=1in,height=1.25in,clip,keepaspectratio]{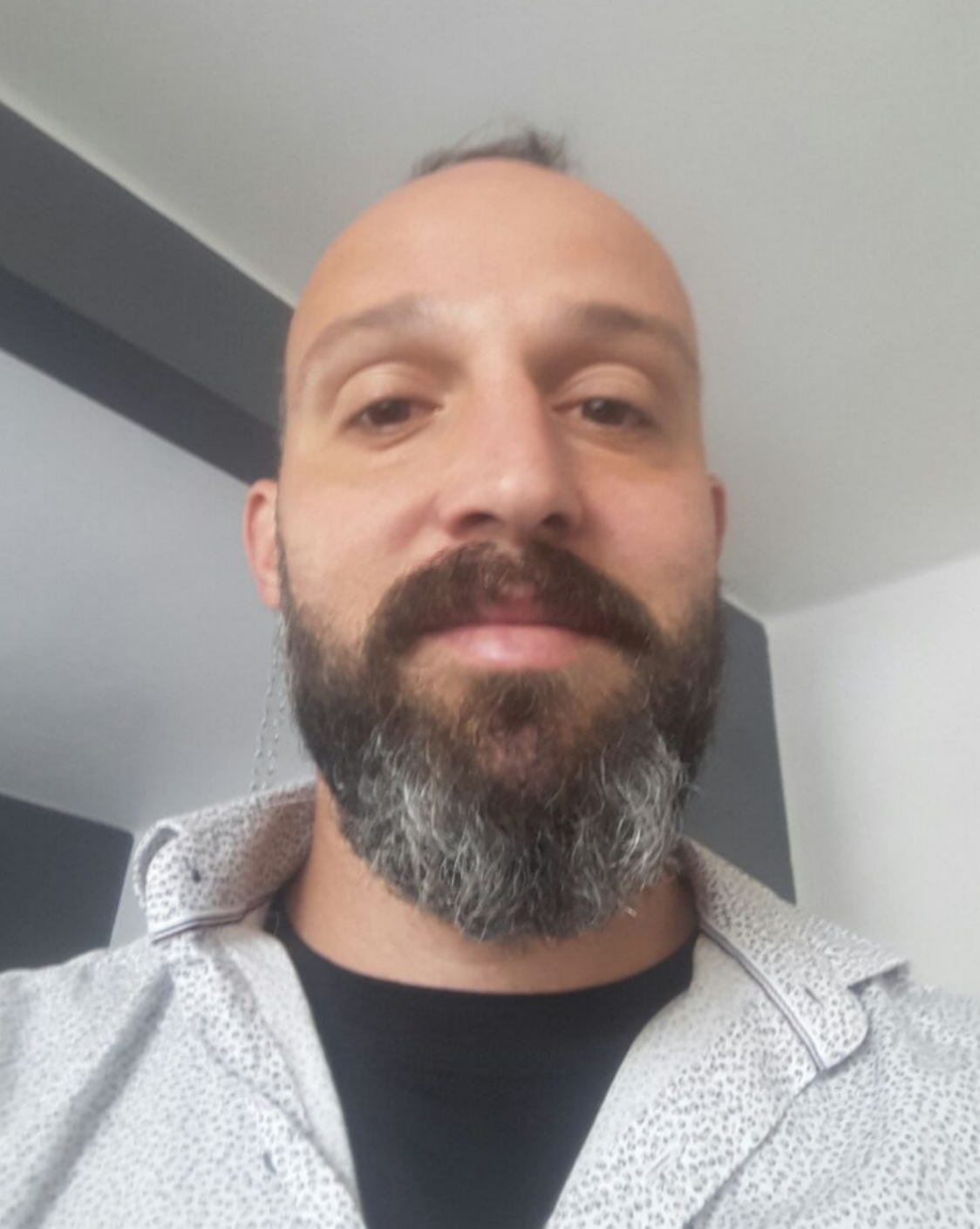}}]{Alessio Franci} got his M.S. (Laurea Specialistica) degree in Theoretical Physics from the University of Pisa in 2008 and his Ph.D. in Physics and Control Theory from the University of Paris Sud XI in 2012. Between 2012 and 2015 he was post-doc in different places (University of Liege, INRIA Lille-Nord Europe, University of Cambridge) but always working at the edge between Control Theory and Neuroscience. Recently, he started to extend his interests to collective decision making in animal groups, neuromorphic circuit design, and spatiotemporal pattern formation in neuroscience and evolutionary biology. He enjoys fruitful interdisciplinary collaborations with researchers in various countries. Since April 2015, he is Associate Professor of Biomathematics at the Universidad Nacional Aut\'onoma de M\'exico, Ciudad de M\'exico, Mexico.
 \end{IEEEbiography}

\begin{IEEEbiography}[{\includegraphics[width=1in,height=1.25in,clip,keepaspectratio]{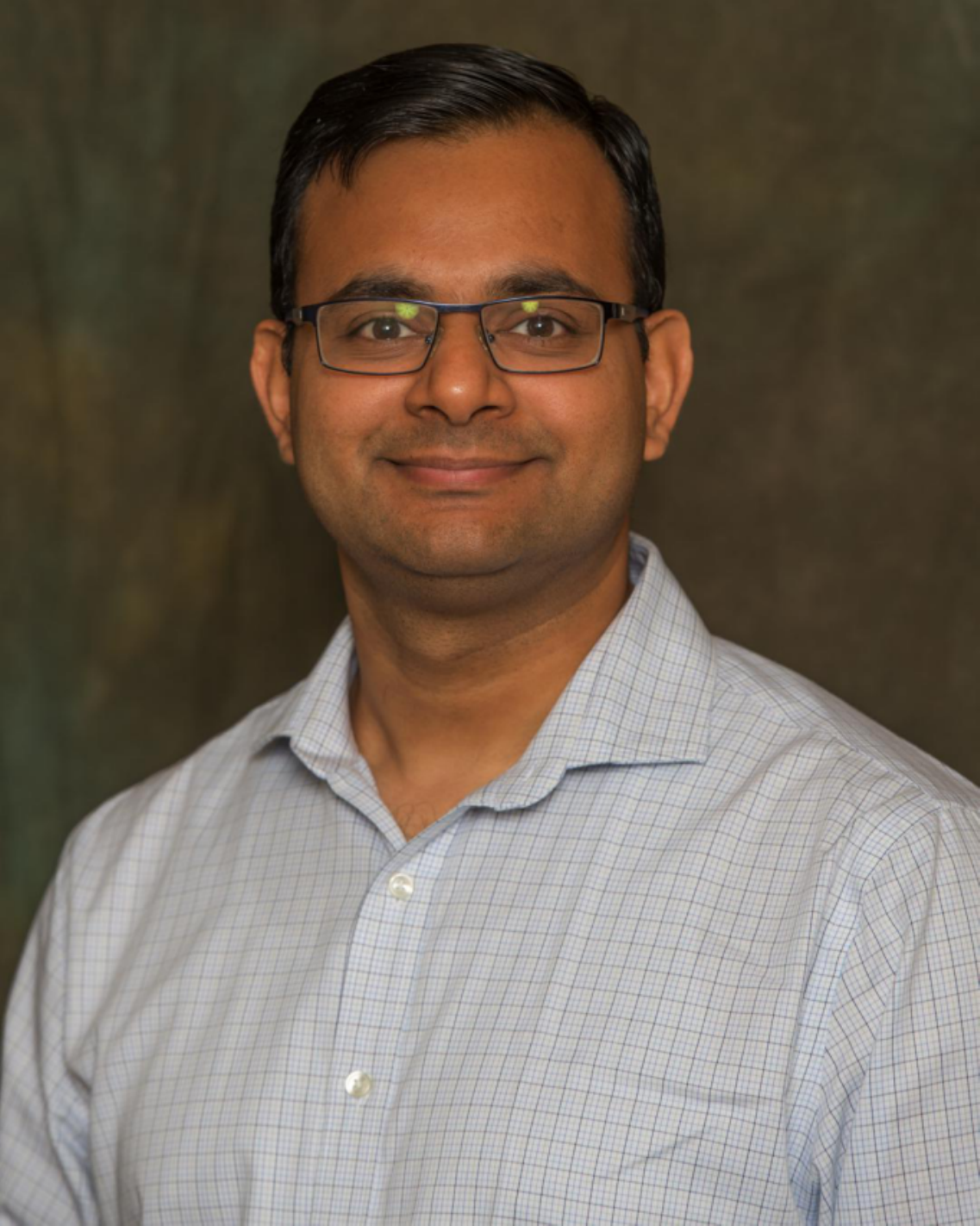}}]{Vaibhav Srivastava} received the B.Tech. degree (2007) in mechanical engineering from the Indian Institute of Technology Bombay, Mumbai, India; the M.S. degree in mechanical engineering (2011), the M.A. degree in statistics (2012), and and the Ph.D. degree in mechanical engineering (2012) from the University of California at Santa Barbara, Santa Barbara, CA. He served as a Lecturer and Associate Research Scholar with the Mechanical and Aerospace Engineering Department, Princeton University, Princeton, NJ from 2013-2016.

Srivastava is an Assistant Professor of Electrical and Computer Engineering at Michigan State University. His research interests include modeling and analysis of human cognition; shared human-automata systems; collective decision-making; and multi-agent robotics.  
 \end{IEEEbiography}

\begin{IEEEbiography}[{\includegraphics[width=1in,height=1.25in,clip,keepaspectratio]{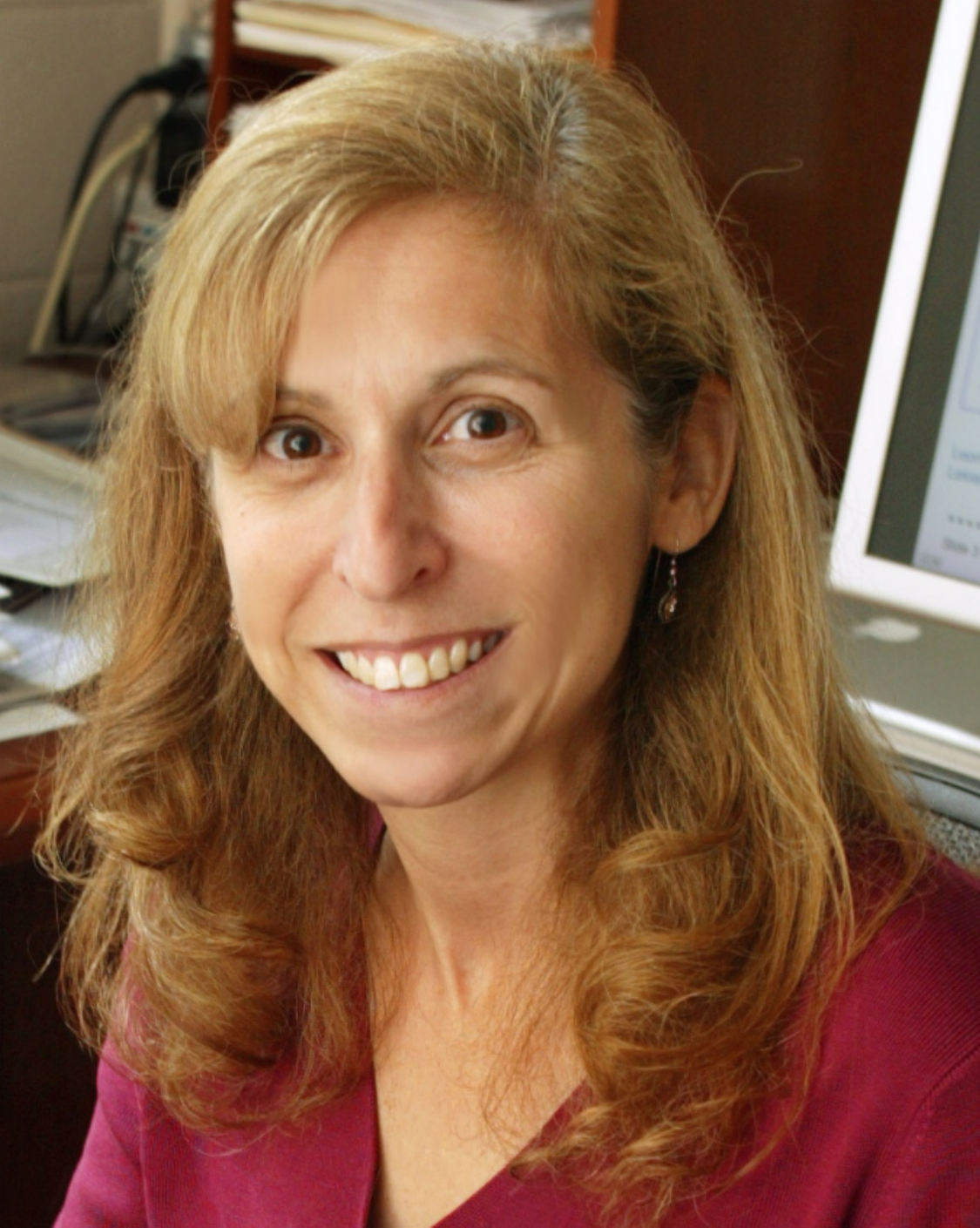}}]{Naomi Ehrich Leonard} (F '07)
received the B.S.E. degree in mechanical engineering from Princeton University, Princeton, NJ, in 1985 and the M.S. and Ph.D. degrees in electrical engineering from the University of Maryland, College Park, in 1991 and 1994, respectively.  From 1985 to 1989, she worked as an Engineer in the electric power industry.  

Leonard is the Edwin S. Wilsey Professor of Mechanical and Aerospace Engineering and Director of the Council on Science and Technology at Princeton University.  She is also  associated faculty of Princeton's Program in Applied and Computational Mathematics.  Leonard's research and teaching are in control and dynamical systems with current interests in coordinated control for multi-agent systems, mobile sensor networks, collective animal behavior, and human decision-making dynamics.  
 \end{IEEEbiography}
\end{document}